\documentclass[11pt,a4paper]{article}
\usepackage{amsmath,amssymb,amsthm,amsfonts,latexsym,graphicx,subfigure}
\usepackage[all,import]{xy}
\usepackage{indentfirst,cite}
\usepackage{tabularx,booktabs}
\usepackage{caption,amscd}
\usepackage{hyperref}\usepackage{extarrows}

\usepackage{fancyhdr,enumerate}
\usepackage{bbm,color}
\usepackage{tikz}
\usetikzlibrary{shapes.geometric, arrows}
\usepackage{accents,cases}
\usepackage{multirow}
\usepackage{euscript}\usepackage{wasysym}\usepackage{pdfsync}\usepackage{comment}

\usepackage{array,float}
\newcommand{\PreserveBackslash}[1]{\let\temp=\\#1\let\\=\temp}
\newcolumntype{C}[1]{>{\PreserveBackslash\centering}p{#1}}
\newcolumntype{R}[1]{>{\PreserveBackslash\raggedleft}p{#1}}
\newcolumntype{L}[1]{>{\PreserveBackslash\raggedright}p{#1}}

\DeclareMathOperator*{\sgn}{\ensuremath{sgn}}

\DeclareMathOperator*{\vol}{\ensuremath{vol}}

\makeatletter
\def\wbar{\accentset{{\cc@style\underline{\mskip8mu}}}}

\makeatother

\renewcommand{\vec}[1]{\mbox{\boldmath \small $#1$}}

\newcommand{\image}{\mathrm{image\,}}

\newcommand{\coloneqq}{{:=}{~}}

\newcommand{\R}{\ensuremath{\mathbb{R}}}

\theoremstyle{plain}
\newtheorem{theorem}{Theorem}[section]

\newtheorem{defn}{Definition}[section]

\newtheorem{lemma}{Lemma}[section]
\newtheorem{remark}{Remark}

\newtheorem{pro}{Proposition}[section]

\allowdisplaybreaks[4]
\begin{document}
\bibliographystyle{unsrt}
\title{Cheeger  inequalities on simplicial complexes
}
%{Discrete-to-Continuous Extensions I: \textsf{k-way Lovasz, one-homogeneous, PL (piecewise linear) extensions}}
%{ Interactions between discrete and continuous worlds via multi-way homogenous extensions }
\author{J\"urgen Jost  
  \and Dong Zhang }

\date{}
\maketitle

\begin{abstract}

 Cheeger-type inequalities in which the decomposability of a graph and the spectral gap of its Laplacian mutually control each other play an important role in graph theory and network analysis, in particular in the context of expander theory. The natural problem to extend such inequalities to simplicial complexes and their higher order Eckmann Laplacians has been open for a long time. Before proving any inequality, however, one needs to identify the right Cheeger-type constant for which such an inequality can hold.  Here, we solve this problem. Our solution involves and combines constructions from simplicial topology, signed graphs, Gromov filling radii and an interpolation between the standard 2-Laplacians and the analytically more difficult 1-Laplacians, for which, however, the inequalities become equalities. 
It is then natural to develop a general theory for  $p$-Laplacians on simplicial complexes and investigate the related Cheeger-type inequalities. %nonlinear version  

\vspace{0.2cm}


\end{abstract}
\tableofcontents

\tikzstyle{startstop} = [rectangle, rounded corners, minimum width=1cm, minimum height=1cm,text centered, draw=black, fill=red!0]
\tikzstyle{io1} = [rectangle, trapezium left angle=80, trapezium right angle=100, minimum width=1cm, minimum height=1cm, text centered, draw=black, fill=blue!0]
\tikzstyle{io2} = [trapezium,  rounded corners, trapezium left angle=100, trapezium right angle=100, minimum width=1cm, minimum height=1cm, text centered, draw=black, fill=yellow!0]
\tikzstyle{process} = [rectangle, minimum width=1cm, minimum height=1cm, text centered, draw=black, fill=orange!0]
\tikzstyle{decision} = [circle, minimum width=1cm, minimum height=1cm, text centered, draw=black, fill=green!0]
\tikzstyle{decision2} = [ellipse, rounded corners=10mm, minimum width=2cm, minimum height=2cm, text centered, draw=black, fill=green!0]
\tikzstyle{arrow} = [thick,->,>=stealth]

\section{Introduction and Background}\label{sec:introduction}

Generalizing the classical Laplace operator, Laplace-type operators have been
defined for functions on various geometric structures, including domains in
Euclidean space or on Riemannian manifolds, and graphs. From their spectra,
one can  usually extract important information about the underlying
structure. In particular, in a seminal paper \cite{Cheeger70}, Cheeger showed
that the first non-vanishing eigenvalue of the Laplace-Beltrami operator of a
compact Riemannian manifold estimates how difficult it is to decompose the
manifold into two pieces. This result has found many generalizations and
extensions, and the discrete analogue, that is, the Cheeger-type inequality
for graphs,  leads to expander theory and  is of fundamental importance in
theoretical computer science and in the analysis of empirical networks when
represented as graphs. And discrete Cheeger-type inequalities can be
generalized, for instance, to weighted or signed graphs, again with diverse
applications. 

But there are also higher-order Laplacians, like the Hodge Laplacian operating
on exterior differential forms on a Riemannian manifold, or its discrete
analogue, the Eckmann Laplacian on a simplicial complex. In this paper, we
look at the simplicial case. 
It is natural to try to generalize the classical spectral results
that are known for graphs to  simplicial complexes. In particular, one can ask
for a version of  the Cheeger inequality for higher dimensional simplicial complexes. 
But it turns out that 
estimating the first non-trivial eigenvalue of the Eckmann Laplacian  on a
simplicial complex is a major long-standing open problem in the field of 
high dimensional expander theory. % We will outline  the state of knowledge,
% and in particular some prior substantial insights, 
% in Section \ref{sec:many-cheeger}.
(Also the analogue in Riemannian geometry, to find  %on establishing 
Cheeger inequalities for  differential $k$-forms, is still far from being
understood and solved.)   And such a Cheeger-type estimate for the
Eckmann Laplacian on a simplicial complex is what we shall develop in this paper. A major difficulty
that we had to overcome consists already in the appropriate formulation of the
inequality, and for that, we need to figure out the relevant aspects of the
combinatorial structure of a simplicial complex that could support such an
inequality. This then needs to be combined with insights coming from a
non-linear analogue of the Laplacian, the 1-Laplacian, which involves the
$L^1$- instead of the $L^2$-norm behind the ordinary Laplacian. That operator is
analytically much more difficult than the ordinary Laplacian, but has the
advantage that the Cheeger-type inequality here becomes an equality. 

Our main result is contained in Theorem \ref{thm:Cheeger-manifold-complex}. It says that there is a constant $C$ such that for all simplicial complexes $\Sigma$ that triangulate a $(d+1)$-dimensional manifold in a uniform manner, 
\begin{equation}\label{1}
\frac{h^2(\Sigma_d)}{C}\le \lambda_{+} \le C\cdot h(\Sigma_d).     \end{equation}
where $\lambda_{+}$ is the first non-vanishing eigenvalue of the $d$-th (normalized) up-Laplacian $\delta^\ast \delta$ (when the corresponding homology is non-trivial, there is a couple of vanishing eigenvalues), and $h(\Sigma_d)$ is our Cheeger constant. The inequality \eqref{1} is of the same form as the known Cheeger inequalities on graphs and Riemannian manifolds, but is new for simplicial complexes.  Moreover, we shall also estimate the spectral gap from $d+2$,  which is the largest possible eigenvalue of the $d$-th normalized up-Laplacian. \\
Nevertheless, we need to recall and assemble some background
material before  our main result can be understood and appreciated. This material will concern
the general setting of Cheeger-type inequalities, 
simplicial  complexes and the Eckmann Laplacian, signed graphs, as well as
$p$-Laplacians, the usual Laplacians corresponding to $p=2$, and the
technically useful case being $p=1$. 

 Although the Cheeger inequalities proved in this paper are inspired by an $L^2$-to-$L^1$ analysis known in graph theory, additional and  essential difficulties occur at least in the proof of Theorem \ref{thm:Cheeger-manifold-complex}. Our approach to overcome these problems  is a combination of (1) a   graph-to-manifold approximation theory on Cheeger cuts \cite{{TSBLB-16},TrillosSlepcev-16,TrillosSlepcev-15,TMT20}, and (2) a nonlinear spectral duality theory developed recently by the authors \cite{{Jost/Zhang21c}}. 
The method is not standard, and we will describe these difficulties and our ideas  in detail in Section \ref{sec:gap-0}.

Our main results then can be described as follows:
\begin{itemize}
\item We establish spectral gaps from the largest possible eigenvalue  by using certain  combinatorial Cheeger-type quantities introduced in Section \ref{sec:gap-d+2}. These results are summarised in  Theorem \ref{thm:anti-signed-Cheeger} as a multi-way Cheeger inequality for estimating spectral gaps from $d+2$, the largest possible eigenvalue of the $d$-th normalized up-Laplacian. 
\item For the spectral gap from $0$, as in \eqref{1}, we propose a new combinatorial Cheeger constant, and we  show that in the most general case, it satisfies a  Cheeger inequality  and is, to the best of our knowledge,   the only  geometric quantity  known to have the same vanishing condition as the first nontrivial eigenvalue of the up-Laplacian.  
We should point out that 
the constants $c$ and $C$ appearing in the Cheeger-type  inequality of the form $ch^2\le \lambda\le Ch$ established in Proposition \ref{pro:rough-Cheeger} depend on the size of the simplicial complex. Importantly, however, as shown in Theorem \ref{thm:Cheeger-manifold-complex}, these constants $c$ and $C$ are  universal
for uniform  triangulations of a $(d+1)$-manifold. 
\end{itemize}

 % Eckmann   introduced a combinatorial operator that is
% analogous to Hodge Laplacian for simplicial complexes. The eigenvalues of all these operators, the Laplace operator on a domain in Euclidean space, the Laplace-Beltrami
% operator on a Riemannian manifold, the Hodge Laplacian on differential forms,  
% and the Eckmann Laplacian on simplicial complexes, all encode important geometric information about the underlying spaces.

% Cheeger inequalities on discrete structures are important objects appearing in expander theory, spectral graph theory and spectral clustering methods. There are many
% generalizations of Cheeger inequalities involving diverse %multifarious  
% objects (e.g., graphs, signed
% graphs, simplicial complexes, hypergraphs, etc) and various  eigenvalue problems (e.g.,
% spectrum of Laplacian, min-max eigenvalues of $p$-Laplacians). 
% In particular, Cheeger inequalities %involving the  Laplacian 
% on simplicial complexes play as a core for  estimating the first nontrivial eigenvalue of the  Eckmann Laplacian. 
% There have been a lot of works on this topic based on a variety of 
% %many kinds of 
% Cheeger-type  constants, which we will outline %overview
% in Section \ref{sec:many-cheeger}. 
% In this paper, we shall mainly propose new Cheeger-type  constants and establish new Cheeger-type  inequalities on simplicial complexes. 
%For a $d$-dimensional  complex, we present both the spectral gap from  and 0
\subsection{Simplicial complexes}\label{sec:Sim-complex}
%\nb

Here, we only consider 
 a finite set $V$ of vertices, leaving the infinite case open. 
%A set-family $\Sigma$ of a set $V$ is an  abstract simplicial complex if it is closed under inclusion, i.e., $\forall \sigma\in\Sigma$, $\forall \sigma'\subset \sigma$, $\sigma'\in\Sigma$.
We recall some standard terminology.  A \emph{simplicial complex} $\Sigma$ on  $V$  is a
    subset of its power set $\mathcal{P}(V)$ that is
    closed under taking subsets, i.e.\   $\forall \sigma\in\Sigma$, $\forall
    \sigma'\subset \sigma$, $\sigma'\in\Sigma$. The elements of $\Sigma$ are
    called \emph{simplices}. %for a \emph{simplex} $C\in   \Sigma$,  any of its subsets $C'\subset C$ is also a simplex in    $\Sigma$. The empty set is  also considered as a simplex.\\
 It follows from this setting that all the vertices constituting a
    simplex are different from each other. A simplex $\sigma$ with $d+1$ vertices is called a 
    \emph{$d$-simplex}, and we call  $d$  its dimension.  Its
    subsimplices are called its \emph{faces}, and its $(d-1)$-dimensional
    faces 
    are called its \emph{facets}.
    The dimension of a simplicial complex  is the largest dimension among its simplices. A
    1-dimensional simplicial complex is a \emph{graph}.

     We usually assume that $\Sigma$  is \emph{connected}. This means that for any
 two of its non-empty simplices $\sigma, \sigma'$, there exists a chain of simplices
 $\sigma_0=\sigma, \sigma_1,\ldots, \sigma_m=\sigma'$ with the property that any two adjacent simplices in this chain
 have at least one vertex in common. And we usually and naturally assume that all elements of the vertex set $V$
  participate in the simplicial complex $\Sigma$, that is, every vertex is
  contained in at least one simplex.
  
In order to work with orientations, we need a slight modification or
amplification of our notation. 
  Here, an \emph{orientation}  of a $d$-simplex  is an ordering of its vertices up to even
  permutation. An odd permutation of the vertices changes an oriented
  $d$-simplex $\sigma_d$ into the oppositely oriented simplex
  $-\sigma_d$. Thus, from now on, $\sigma_d$ denotes an ordered simplex. 
    
    Let $\Sigma_d$ be the collection of the
    $d$-simplices of $\Sigma$.  In  particular, $\Sigma_0$ is the vertex set
    $V$. We let $C_d=C_d(\Sigma)$ be the abelian group with
    coefficients in $\R$ generated by the elements of $\Sigma_d$. We also
    write $C^d=C^d(\Sigma)$ for the linear functions from $C_d$ to $\R$ that satisfy
 \begin{equation}
  \label{in11}
  f(-\sigma_d)=-f(\sigma_d),
\end{equation}
  for every oriented $d$-simplex.

 For $f\in C^{d-1}$, we  define its \emph{coboundary} $\delta f:C^{d}\to \R$ as
 \begin{equation}
\label{in6}
\delta f(v_0,v_1,\ldots ,v_d)=\sum_{i=0}^d (-1)^if(v_0,\ldots, \hat{v_i},\ldots ,v_d),
\end{equation}
where, as usual,  a $\hat{\,}$ over a vertex means that it is omitted. 
Sometimes, we write
\begin{equation}\label{in7}
\delta_d: C^d \to C^{d+1},
\end{equation}
in order to specify the dimension.

The $d$-th \emph{cohomology group}\index{cohomology group} of the  simplicial complex $\Sigma$  is
\begin{equation}
\label{ch5a}
H^d(\Sigma):= \ker \delta_d/\image \delta_{d-1}.  
\end{equation}
%The dimension $b_d(\Sigma)$ of $H^d(\Sigma)$ is called the $d$-th \emph{Betti  number} of $\Sigma$.

  \begin{remark}
More generaly, we can consider  the linear space $C_d(\Sigma,\mathbb{F})$ with
    coefficients in an abelian group  $\mathbb{F}$, generated by the elements of $\Sigma_d$, and let 
   $C^d(\Sigma,\mathbb{F})$ be the linear functions from
   $C_d(\Sigma,\mathbb{F})$ to $\mathbb{F}$, satisfying \eqref{in11},  and then we can define the  cohomology group  $H^d(\Sigma,\mathbb{F})$ in the same way. It is usual to take $\mathbb{F}$  to be a commutative ring (e.g.\ the integer ring $\mathbb{Z}$) or even a field (e.g.\ the  field  $\mathbb{C}$ of the complex numbers, or the finite field $\mathbb{Z}_p\coloneqq \mathbb{Z}/p\mathbb{Z}$). As an interesting example, we refer to \cite{Steenbergen14} for the Cheeger constants defined on a simplicial complex which use the cohomology over the finite field $\mathbb{Z}_2$. In this paper, we  work with  real coefficients   or integer coefficients.

 If we pass to the reduced cochain complex, we get the reduced cohomology $\tilde{H}^d$, which can be  defined simply by the relation $\tilde{H}^0(\Sigma,\mathbb{F})\oplus \mathbb{F}=H^0(\Sigma,\mathbb{F})$ and $\tilde{H}^d(\Sigma,\mathbb{F})=H^d(\Sigma,\mathbb{F})$ for $d\ge1$, where $\mathbb{F}$ can be $\mathbb{Z}_2$, 
$\mathbb{Z}$ or $\mathbb{R}$.
\end{remark}
  
% We consider a simplicial  complex $\Sigma$. 
% We recall that, concerning the orientations, according to \eqref{in11}, we have
% for any $\phi \in C^d$, where $C^d$  are the linear functions on chains of
% $d$-dimensional 
% (hyper)simplices, 
% \begin{equation}
% \label{lap0}
% \phi(-\sigma_d)=-\phi(\sigma_d),
% \end{equation}
% that is, changing the orientation yields a minus sign. 

To proceed, we choose  positive definite inner products $(\cdot,\cdot)_d$
 on the  $C^d$. We can then define the adjoint
$(\delta_{d})^{*}:C^{d+1}\rightarrow ~C^{d}$  of the
coboundary operator $\delta_{d}$   by 
 $$
 (\delta_{d}f_{1},f_{2})_{d+1}=(f_{1},(\delta_{d})^{*} f_{2})_{d},
 $$
for  $f_{1}\in C^{d}$ and $f_{2}\in C^{d+1}$. We can then go back and forth between the $C^d$, as we have  the  arrows 
\begin{equation} \label{lap1}
C^{d-1}
% use packages: array
\begin{array}{l}
\underrightarrow{\delta_{d-1\textrm{ }}}\\ 
 \overleftarrow{ {\delta_{d-1}}^*}
\end{array}
C^{d} \begin{array}{l}
\underrightarrow{\textrm{ }\textrm{ }\delta_{d\textrm{ }\textrm{ }}}\\ 
 \overleftarrow{ {\textrm{ }\textrm{ }\delta_{d}}^{* \textrm{ }}}
\end{array}
 C^{d+1}.
\end{equation}
This allows us to define the following three operators on $C^{d}$ (omitting
the argument $\Sigma$, i.e., writing for instance $L_d$ instead of
$L_d(\Sigma)$, as $\Sigma$ will be mostly kept fixed):
 \begin{enumerate}
\item[(i)] The \emph{$d$-dimensional up Laplace operator}  or simply \emph{$d$-up Laplacian} of the simplicial complex  $\Sigma$ is 
$$
L_{d}^{up}\coloneqq (\delta_{d})^{*}\delta_{d},
$$
\item [(ii)]The \emph{$d$-dimensional  down Laplace operator}  or  \emph{$d$-down Laplacian} is
$$
L_{d}^{down}\coloneqq \delta_{d-1}(\delta_{d-1})^{*},
$$
\item [(iii)]The \emph{$d$-dimensional Laplace operator}  or
  \emph{$d$-Laplacian} is the sum 
$$
L_{d}\coloneqq   L_{d}^{up}+ L_{d}^{down} =(\delta_{d})^{*}\delta_{d}+\delta_{d-1}(\delta_{d-1})^{*}.
$$
\end{enumerate}
 The operators $L_{d}^{up}$, $L_{d}^{down}$ and  $L_{d}$  are  self-adjoint
 and non-negative.  Therefore, their eigenvalues are non-negative real
 numbers.  

The multiplicities of the eigenvalue $0$ of the
Laplacians $L_d(\Sigma)$ contain topological information about $\Sigma$. This is the content of
Eckmann's Theorem  \cite{Eckmann44},
which is a  discrete version of the Hodge theorem. It says that 
 \begin{displaymath}
  \ker L_{d}(\Sigma)\cong {H}^{d}(\Sigma).
 \end{displaymath}
Thus, the multiplicity of the eigenvalue $0$ of the operator $L_d(\Sigma)$
is equal to the Betti number $b_d$, the dimension of ${H}^{d}(\Sigma)$. As a corollary, 
\begin{equation}
    C^d=\image \delta_{d-1}\oplus  \image (\delta_d)^{*} \oplus \ker L_d.
  \end{equation}
  We point out that Eckmann's Theorem does not depend on the choice of scalar products on the spaces $C^d$ (although the harmonic cocycles do).\\
While cohomology groups are defined as quotients, that is, as equivalence
classes of elements of  $C^d$, Eckmann’s Theorem provides us with concrete
representatives in $C^d$ of those equivalence classes, the so-called
\emph{harmonic cocycles}. These are the eigenvectors for  
  the eigenvalue $0$ of the Laplacian. We shall now look at  the non-zero part
  of the spectrum which will depend on the choice of the scalar products. 

Since $\delta_{d}\delta_{d-1}=0$ and ${\delta_{d-1}}^{*}{\delta_{d}}^{*}=0$, %(recall \eqref{lap1}), 
\begin{align}
&\image L_{d}^{down}(\Sigma) \subset  \ker L_{d}^{up}(\Sigma)\label{hodge},\\
&\image  L_{d}^{up}(\Sigma) \subset  \ker L_{d}^{down}(\Sigma)\label{hodge1}.
\end{align}
This implies that 
$\lambda\neq 0$ is an eigenvalue of $L_d(\Sigma)$ if and only if it is a  eigenvalue of either $L_{d}^{up}(\Sigma)$ or  $L_{d}^{down}(\Sigma)$. 
Therefore, the non-zero parts of the spectra satisfy 
\begin{equation}
\label{com1}
\mathop{\mathrm{spec}}\limits_{\neq 0}(L_{d}(\Sigma))= \mathop{\mathrm{spec}}\limits_{\neq 0}(L_d^{up}(\Sigma))\cup \mathop{\mathrm{spec}}\limits_{\neq 0}(L_d^{down}(\Sigma)).
\end{equation}
The multiplicity of the eigenvalue $0$ may be different, however.

Since $\mathop{\mathrm{spec}}\limits_{\neq 0}(AB)= \mathop{\mathrm{spec}}\limits_{\neq 0}(BA)$, for linear
operators $A$ and $ B$  on  Hilbert spaces, we conclude
\begin{equation}
\label{com2}
\mathop{\mathrm{spec}}\limits_{\neq 0}(L_{d}^{up}(\Sigma))=\mathop{\mathrm{spec}}\limits_{\neq 0}(L_{d+1}^{down}(\Sigma)).
\end{equation}
From (\ref{com1}) and (\ref{com2}) we conclude that each of  the three families of multisets 
$$\{\mathop{\mathrm{spec}}\limits_{\neq 0}(L_{d}(\Sigma))\mid 0\leq d \leq m\}\textrm{,}\; \{\mathop{\mathrm{spec}}\limits_{\neq 0}(L_{d}^{up}(\Sigma))\mid 0\leq d \leq m-1\}\textrm{,} \; \{\mathop{\mathrm{spec}}\limits_{\neq 0}(L_{d}^{down}(\Sigma))\mid 1\leq d \leq m\}$$
determines the other two. Therefore, it suffices to consider only one
of them.

We shall also make use of the following general result, the
Courant-Fischer-Weyl minimax principle. 
\begin{lemma}\label{rayleigh}
  Let the linear operator $A:H\to H$ on a finite dimensional vector space be
  self-adjoint w.r.t. the scalar product $(.,.)$. Then its eigenvalues and
  eigenvectors are the critical values and the critical points of the Rayleigh
  quotient, defined for $f\neq 0$, 
  \begin{equation}
    \label{r1}
    \frac{(Af,f)}{(f,f)}. 
  \end{equation}
\end{lemma}

\subsection{Cheeger inequalities}
\label{sec:many-cheeger}

Cheeger\cite{Cheeger70} showed that the first non-vanishing eigenvalue of the
Laplace-Beltrami operator of a compact connected Riemannian manifold can be bounded from below in
terms of a constant introduced by him and thence called the Cheeger constant
which quantifies how difficult it is to cut the manifold into two large pieces
by a small hypersurface. Buser \cite{Buser} then also showed an upper
estimate. Thus, this eigenvalue is controlling and controlled by the
Cheeger constant. It was then realized in \cite{Dodziuk84,Alon,Chung} that an
analogous estimate holds on graphs, for the first non-vanishing eigenvalue of
the graph Laplacian. The analogue of Cheeger's constant had in fact already been introduced
by Polya \cite{Polya}, without  connecting it to eigenvalues. To formulate the
latter inequalities, we consider an undirected and unweighted  graph $\Gamma=(V,E)$ with vertex set $V$
and edge set $E$. The degree $\deg v$ of a vertex $v$ is the number of its
neighbors, that is, the number of vertices directly connected it by edges.
We define the volume of $S\subset V$ is
$\vol (S)= \sum_{v\in S} \deg v$,  
for $V_1,V_2\subset V$,
$|E(V_1,V_2)|$ is the number of edges with one endpoint in $V_1$ and the
other in $V_2$. 
We then put
\begin{equation}\label{che17a}
\eta(S):=\frac{|{\partial S}%E(S,V\backslash  S)
  |}{\min(\mathrm{vol}(S),\mathrm{vol}(V\backslash
  S))},
\end{equation}
and introduce the \emph{(Polya)-Cheeger constant}
\begin{equation}\label{che17b}
h= \min_{S\ne\emptyset,V}\eta(S).
\end{equation}
The estimate for the first non-vanishing eigenvalue $\lambda$ of the
normalized graph Laplacian then says
\begin{equation}
  \label{che17c}
  \frac{1}{2} h^2\le \lambda \le 2h. 
\end{equation}
This estimate is important, for instance, in the theory of expander graphs,
because a good expander  family  should have a  uniformly large such $\lambda$.

In fact, one can not only bound the smallest non-vanishing eigenvalue of a
graph from below, but also the largest one from above. The largest eigenvalue
of the normalized Laplacian of a graph is always $\le 2$. Equality is realized
for bipartite graphs, and for non-bipartite graphs, the difference $2-\lambda$
can be controlled \cite{Bauer,Trevisan}. 

Already in the original paper by Cheeger \cite{Cheeger70}, the problem was
proposed to derive an estimate for the smallest non-vanishing eigenvalue of
the Hodge Laplacian on differential $k$-forms on a
closed Riemann manifold. So far, this problem is not solved. Its discrete
version, a Cheeger-type  inequality on simplicial complexes, is also a
long-standing open problem in the area of high dimensional expanders.
  While some partial answers have been proposed and developed in the literature, it seems that none of them gives a complete solution to the problem.  In fact, there are many different definitions of  Cheeger  constants on
simplicial complexes. For example, it is known that the  easier upper bound of \eqref{che17c}   holds for the Cheeger constant suggested by Parzanchevski, Rosenthal and Tessler \cite{Gundert14,Parzanchevski15}. But  none of the constants proposed so far in the literature   can satisfy a  full  Cheeger  inequality as in the
graph setting. In particular, in the field of higher-dimensional expanders,
people use the so-called  $\mathbb{Z}_2$-expander for constructing  the  Cheeger
constants on a simplicial complex  (see \cite{LM06,MW09,Dotterer12,Gundert12,Gromov} and papers which stem from them for $\mathbb{Z}_2$-cohomological Cheeger constants).

 Thus, the problem is, and the essential purpose of this paper  is to establish some good  
 estimate for the first nontrivial eigenvalue of the discrete Eckmann Laplacian
 by introducing some suitable Cheeger-type constants on a simplicial complex
 and proving that this controls, and in turn is controlled by that
 eigenvalue, analogously to \eqref{che17c}. 
 Controlling this eigenvalue from below in terms of the Cheeger-type constant
 is called the Cheeger side, while controlling it from above is called the
 Buser side. Usually, the latter is easier than the former.

Important contributions in this direction come from  Dotterrer and Kahle
\cite{Dotterer12} and Steenbergen, Klivans and  Mukherjee
\cite{Steenbergen14}. 
In \cite{Steenbergen14},  a   Cheeger constant via cochain
complexes is analyzed, 
\begin{equation}
  \label{skm}
h^d(\Sigma):=
\min\limits_{\phi\in
  C^d(\Sigma,\mathbb{Z}_2)\setminus\mathrm{Im\,}\delta}\frac{\|\delta\phi\|}{\min\limits_{\psi\in\mathrm{Im\,}\delta}\|\phi+\psi\|}
\end{equation}
 which satisfies
$$h^d(\Sigma)=0\Longleftrightarrow \tilde{H}^d(\Sigma,\mathbb{Z}_2)\ne0,\;\; \forall d\ge 0,$$
where $\|\cdot\|$ is the Hamming norm on $C^d(\Sigma,\mathbb{Z}_2)$ (i.e.\ the
$l^1$-norm on $\mathbb{Z}_2^{n}$ with $n=\#\Sigma_d$).
 This 
is a natural generalization of the classical graph Cheeger constant
\eqref{che17b} to higher dimensions on simplicial complexes.  
Unfortunately, based on the results in \cite{Gundert12}  and \cite{Steenbergen14}, 
the most straightforward attempt at a higher-dimensional
Cheeger inequality fails, even for the Buser side -- in higher
dimensions, spectral expansion (an eigenvalue gap for the
Laplacian) does not imply combinatorial expansion  \cite{Lubotzky18}. 
In fact, according to the
examples and theorems in
\cite{Dotterer12,Gundert14,Gundert16,Parzanchevski15,Steenbergen14}, all the
 Cheeger constants defined using cohomology (or homology) with
$\mathbb{Z}_2$-coefficients cannot satisfy a general two-sided Cheeger
inequality as in the graph setting. This is a consequence of the  relation
$$\lambda(\Delta_d^{up})=0\Leftrightarrow{\lambda(L_d^{up})=0
  \Longleftrightarrow \tilde{H}^d(\Sigma,\R)\ne0},\;\; d\ge 0,$$
but for $d\ge 1$, the non-vanishing of $\tilde{H}^d(\Sigma,\R)$ is not equivalent
to that of $\tilde{H}^d(\Sigma,\mathbb{Z}_2)$.\\

We should also mention that in \cite{Parzanchevski15}, another Cheeger-type constant is
proposed, and their Theorem 1.2 generalizes the  upper
Cheeger inequality to higher dimensions. 
That modified
Cheeger number  is nonzero only if the simplicial complex has a complete skeleton,
and the Cheeger side of the inequality includes an additive constant. 
We shall adopt a different definition, and therefore do not go into further
detail here.

In this paper, we shall first derive 
Theorem \ref{thm:anti-signed-Cheeger} which contains an estimate for the spectral gap from %with 
$d+2$, recalling that for the vertex Laplacian of a  graph, i.e., in the case
$d=0$, the spectral gap
at 2 can be controlled.  %which should be seen as the other end of the spectrum of $\Delta^{up}_d$. 
We then turn to the more difficult   estimate  for the spectral gap from %with 
$0$, namely, the Cheeger-type estimate for the first non-trivial eigenvalue of
the Eckmann Laplacian. 
Since such an estimate cannot be derived for the Cheeger-type constants
introduced earlier, our first contribution here is the  introduction of  a new
Cheeger  constant. The key point is that in contrast to the graph case, on
higher dimensional simplices, orientations and multiplicities enter into the coboundary relations
  and therefore implicitly into the
  eigenvalues. We therefore consider generalized (i.e., with both positive and negative multiplicities) multisets
  of $d$-simplices.

\subsection{Signed graphs}\label{signed}
We consider unweighted and undirected graphs $\Gamma$. When $v,v'\in V$, the vertex
set of  $\Gamma$, are connected by an edge, denoted as $(vv')$, we write $v\sim v'$ and
call $v$ and $v'$ neighbors. We shall need an additional structure, a sign
function on the edges. A  \emph{signed graph} thus is a graph $\Gamma$ equipped with a map $s$ from its edge set
to $\pm 1$. We may \emph{switch signs} by taking a vertex and changing the signs of all edges that it is
contained in. A signed  graph is called \emph{balanced} if by switching some vertices, we can make
all signs $=1$, and it is \emph{antibalanced}, if we can make them all $=-1$.

Signed  graphs have many   applications in modeling  biological networks, social relations, ferromagnetism, and general signed networks \cite{Zaslavsky,Harary53, ArefWilson19,ArefMasonWilson20}. 
The spectral theory for signed graphs has led to a number of breakthroughs in theoretical computer science and combinatorial geometry,  including the solutions to the  sensitive conjecture \cite{Huang19} and the open problem on equiangular lines \cite{JTYZZ21,JTYZZ}.

%Using the framework of signed graphs developed  in Section \ref{signed},  
The  Laplacian  of the signed graph $(\Gamma,s)$ is
\begin{equation}
  \label{slap}
  \Delta_s f(v)= f(v)-\frac{1}{\deg v}\sum_{v' \sim v}s(vv')f(v')=\frac{1}{\deg v}\sum_{v' \sim v}(f(v)-s(vv')f(v'))
\end{equation}
We record some basic results about the spectrum of this operator
\cite{Atay20} that can be easily checked.
\begin{lemma}\label{ss}
The eigenvalues of $\Delta_s$ are real and lie in the interval $[0,2]$. In
fact, the smallest eigenvalue is $=0$ if and only if $(\Gamma,s)$ is balanced,
and positive otherwise. Likewise, the largest eigenvalue is $=2$ if and only
if the graph is antibalanced.
\end{lemma}

To proceed, we recall  %study 
the multi-way  Cheeger constant  $h_k^s$ on a signed graph $(\Gamma,s)$ introduced in \cite{Atay20}. 
For  disjoint $V_1,V_2\subset V$, let $E^+(V_1,V_2)=\{\{u,v\}\in E:u\in V_1,v\in V_2,s(uv)=1\}$ and $E^-(V_1)=\{\{u,v\}\in E:u,v\in V_1,s(uv)=-1\}$. The  signed bipartiteness ratio is defined as
$$\beta^s(V_1,V_2)=\frac{2\left(|E^-(V_1)|+|E^-(V_2)|+|E^+(V_1,V_2)|\right)+|\partial(V_1\sqcup V_2)|}{\vol(V_1\sqcup V_2)}.$$
The  signed Cheeger constant %(see also  \eqref{sign40}) 
of the signed graph $(\Gamma,s)$ is then defined as $$h^s=\min\limits_{(V_1,V_2)\ne(\emptyset,\emptyset)}\beta^s(V_1,V_2)$$
where the minimum is taken over all possible sub-bipartitions of
$V$. $\beta^s$ and hence also $h^s$ is switching invariant.  

The Cheeger inequality for signed graphs established in \cite{Atay20} says that 
for a  signed graph $(\Gamma,s)$, we have 
\begin{equation}
  \label{atay}
\frac{\lambda_1(\Delta_s)}{2}\le h^s\le \sqrt{2\lambda_1(\Delta_s)}.  
\end{equation}

The $k$-way signed Cheeger constant is defined as
$$h_k^s=\min\limits_{\{(V_{2i-1},V_{2i})\}_{i=1}^k}\max\limits_{1\le i\le k}\beta^s(V_{2i-1},V_{2i})$$
where the minimum is taken over the set of all possible $k$  pairs of
disjoint sub-bipartitions $(V_1,V_2)$, $(V_3,V_4)$, $\ldots$,
$(V_{2k-1},V_{2k})$.  $h^s_k$ is again switching invariant.  

This definition allowed Atay and Liu %we can 
to generalize and put into perspective the higher-order Cheeger inequality for
ordinary graphs by Lee, Oveis Gharan, and Trevisan \cite{LGT12}. Their
estimate is 
\begin{theorem}[\cite{Atay20}]\label{thm:signed-Cheeger}
 There exists an absolute constant $C$ such that for any
signed graph $(\Gamma,s)$, and any $k\in\{1,\ldots,n\}$, $$\frac{\lambda_k(\Delta_s)}{2}\le h_k^s\le Ck^3 \sqrt{\lambda_k(\Delta_s)}.$$ 
\end{theorem}

\subsection{A relation between simplicial complexes and signed graphs}\label{sec:simp-sign}

%As explained in Section \ref{sec:Sim-complex}, w
Formally, we shall work on an abstract
simplicial complex $\Sigma$ with  vertex set $V=\{1,\cdots,n\}$. For
$\sigma=\{i_0,\cdots,i_d\}\in \Sigma$, we use $[\sigma]:=[i_0,\cdots,i_d]$ to
indicate  the  oriented  $d$-dimensional simplex which is  formed by $\sigma$
when arranging its vertices in the specified order. Now, we choose  an ordering on the  vertices of each $d$-simplex, i.e., we fix an orientation of each simplex. We then let  $[\Sigma_d]=\{[\sigma]:\sigma\in \Sigma_d\}$ be the set of the  oriented  $d$-simplexes,  
%Two orderings of the vertices are said to determine the same orientation if there is an even permutation transforming one ordering into the other. If the permutation is odd, then the orientations are opposite.  
in which we have fixed an arbitrary orientation for each simplex. 

Analogously to the cochain group $C^d(\Sigma)$, the $d$-th chain group $C_d(\Sigma)$ of $\Sigma$ is a vector space with the base
 $[\Sigma_d]$. The boundary map $\partial_d:C_d(\Sigma)\to C_{d-1}(\Sigma)$  is a linear operator 
 defined by 
 \[\partial_d[i_0,\cdots,i_d]=\sum_{j=0}^d(-1)^j[i_0,\cdots,i_{j-1},i_{j+1},\cdots,i_d],\]
 which can  also be represented by the incidence matrix $B_d$ of dimension
 $|\Sigma_{d-1}|\times |\Sigma_d|$ whose  elements belong to  $\{-1,0,1\}$.

 With this notation, the $d$-th cochain group $C^d(\Sigma)$ is  the dual of
 the chain group $C_d(\Sigma)$, that is, $C^d(\Sigma)$ collects all the $\R$-valued skew-symmetric function on all ordered $d$-simplices. The  simplicial coboundary map
 $\delta_d:C^d(\Sigma)\to C^{d+1}(\Sigma)$ is a linear operator generated by
 $(\delta_df)([i_0,\cdots,i_{d+1}])=\sum_{j=0}^{d+1}(-1)^jf([i_0,\cdots,i_{j-1},i_{j+1},\cdots,i_{d+1}])$
 for any $f\in C^d(\Sigma)$. It is obvious that $\delta_d=B_{d+1}^\top$, and
 we can then define the adjoint via $\delta_d^*=B_{d+1}$. We
 can therefore  also use the incidence matrices  to express the Laplace operators (see \cite{Horak13a}): 
\begin{enumerate}[-]
\item the $d$-th up Laplace  operator  $L^{up}_d=\delta_d^*\delta_d=B_{d+1} B_{d+1}^\top$ 
\item the $d$-th down Laplace  operator  $L^{down}_d=\delta_{d-1}\delta_{d-1}^*=B_d^\top B_d$ 
\item the $d$-th   Laplace  operator  $L_d=L^{up}_d+L^{down}_d=\delta_d^*\delta_d+\delta_{d-1}\delta_{d-1}^*=B_d^\top B_d+B_{d+1} B_{d+1}^\top$ 
\end{enumerate}

The aim  of the present paper is to provide new Cheeger-type inequalities for
the first nontrivial eigenvalues of $L_d$,  $L_{d}^{up}$ and
$L_{d}^{down}$. As explained in Section \ref{sec:Sim-complex}, it
suffices to consider 
 $L_{d}^{up}$ for every $d$. And as stated in that section, this
 operator depends on the choice of scalar products. With an appropriate
 choice, we obtain the \emph{normalized} Laplacian, which we denote by
 $\Delta_{d}^{up}$, in order to distinguish it from the general case. 
 It is given by
 \[(\Delta_{d}^{up} f)([\sigma])=  f([\sigma]) 
+  \frac{1}{\deg \sigma}\sum_{\substack{\sigma'\in \Sigma_{d}: \sigma'\neq \sigma,\\ \exists!\rho\in\Sigma_{d+1}\text{ s.t. }\sigma,\sigma'\text{ are facets of } \rho }} \sgn([\sigma],\partial [\rho])\sgn([\sigma'],\partial [\rho])f([\sigma']),  \]
%We now apply the constructions of Section \ref{signed}. 
Our results will be obtained for this operator, and they only partially
generalize to a general $L_{d}^{up}$.

A key step is to 
express the up-Laplacian of a  simplicial complex in terms of the Laplacian of an associated signed graph. 
%The results here are mainly taken from \cite{Jost/Zhang21c}. 
%We recall the formula \ref{lap30} for 
The  normalized up-Laplacian of a  simplicial complex $\Sigma$ can be written as
\begin{equation}
\label{asc1}
(\Delta_{d}^{up} f)([\sigma])=  f([\sigma]) 
-  \frac{1}{\deg \sigma}\sum_{\substack{\sigma'\in \Sigma_{d}: \sigma'\neq \sigma,\\ \exists!\rho\in\Sigma_{d+1}\text{ s.t. }\sigma,\sigma'\text{ are facets of } \rho }} s([\sigma],[\sigma'])f([\sigma']),    
\end{equation}
where we have put
\begin{align}
s([\sigma],[\sigma'])&:= 
-\sgn([\sigma],\partial [\rho])\sgn([\sigma'],\partial [\rho]).\label{asc2}
%\ea        
\end{align}

Thus, we may express $\Delta_{d}^{up}$ in terms of  the Laplacian
$\Delta_{(\Gamma_d,s)}$ for the signed graph $(\Gamma_d,s)$ with vertex set
consisting of the $d$-simplices of our simplicial complex, and where two
different such vertices $\sigma, \sigma'$ are connected by an edge, $\sigma
\sim \sigma'$, if there exists a $(d+1)$-simplex $\rho$ in $\Sigma$ with
$\sigma,\sigma'\in\partial  \rho$. 

\begin{remark}
This construction is very natural and  essentially follows from the definition
of the (up/down)  combinatorial  Laplacian matrices of a simplicial complex. A similar idea was already used to define the signed adjacency matrix of a triangulation on a surface \cite{FST08}. 
\end{remark} 
The relation between the up-Laplacian and the signed graph Laplacian
\eqref{slap} is
\begin{equation}
 \label{asc3}
\Delta_{d}^{up}= (d+1)\Delta_{(\Gamma_d,s)} - d\ \mathrm{Id}.
\end{equation} 
By  \eqref{asc3}, the eigenvalues $\mu_j$ of $\Delta_{d}^{up}$ and the eigenvalues $\lambda_j$ of $\Delta_{(\Gamma_d,s)}$ are related by
\begin{equation}
    \label{asc5}
\mu_j =(d+1)\lambda_j -d.
\end{equation}
Since the eigenvalues of $\Delta_{(\Gamma_d,s)}$ lie in the interval $[0,2]$,
those of 
   $\Delta_{d}^{up}$ lie in the interval $[0,d+2]$.
In fact, since $\mu_j \ge 0$ in \eqref{asc5}, the eigenvalues of $\Delta_{(\Gamma_d,s)}$ are $\ge  \frac{d}{d+1}$. Equality
  holds if and only if there is some non-trivial $f$ with $\delta_d
  f=0$. More precisely, the multiplicity of the eigenvalue $\frac{d}{d+1}$ of
  $\Delta_{(\Gamma_d,s)}$ equals the dimension of the kernel of the coboundary
  operator $\delta_d$. 
In particular, for $d>0$, the graph $(\Gamma_d,s)$ is never balanced. 

The next result then is an easy consequence of Lemma \ref{ss}. 
\begin{pro}\label{prop:q+2 eigen}
The spectrum of $\Delta_{d}^{up}$ contains the  eigenvalue $d+2$ if and only if the signed graph $(\Gamma_d,s)$ has an antibalanced component. Moreover, the multiplicity of  $d+2$ equals the number of antibalanced components of $(\Gamma_d,s)$.
\end{pro}

We consider the opposite $(\Gamma_d,-s)$ of the signed graph $(\Gamma_d,s)$,
with its Laplacian  $\Delta_{(\Gamma_d,-s)}$. Then,  the eigenvalues of the three Laplacians  $\Delta_{d}^{up}$, $\Delta_{(\Gamma_d,s)}$ and $\Delta_{(\Gamma_d,-s)}$ satisfy  the relation:
$$
\begin{matrix}
\text{Spectrum of }\Delta_{d}^{up} & ~& \text{Spectrum of }\Delta_{(\Gamma_d,s)} & ~& \text{Spectrum of }\Delta_{(\Gamma_d,-s)} \\
\\
   &~&   &~&    \\
 0 &~&  \frac{d}{d+1}  &~&  \frac{d+2}{d+1} \\
\vdots  &~& \vdots  &~&  \vdots \\ \lambda &\Longleftrightarrow& \frac{\lambda+d}{d+1}  &\Longleftrightarrow&  \frac{d+2-\lambda}{d+1} \\
\vdots  &~& \vdots  &~&  \vdots \\ d+2 &~& 2 &~&  0 \\
\end{matrix}
$$
that is,
\begin{pro}
$\lambda$ is an eigenvalue of $\Delta_{d}^{up}$ if and only if
$\frac{\lambda+d}{d+1} $ is an eigenvalue of $\Delta_{(\Gamma_d,s)}$ if and
only if $\frac{d+2-\lambda}{d+1} $ is an eigenvalue of
$\Delta_{(\Gamma_d,-s)}$.
\end{pro}

In addition, analogously to Proposition  \ref{prop:q+2 eigen}
\begin{pro}
The multiplicity of the eigenvalue $0$ of $\Delta_{d}^{up}$ is
$\ge d+1$ (when the simplicial complex is pure, the multiplicity of the  eigenvalue $0$ of $\Delta_{d}^{up}$ is $d+1$ if and only if the simplicial complex is a simplex of dimension $d+1$). 
And the multiplicity of the eigenvalue $d+2$ of $\Delta_{d}^{up}$ agrees  with
the number of balanced components of $\Delta_{(\Gamma_d,-s)}$.
\end{pro}

\subsection{$p$-Laplacians}\label{plap}
An essential feature of Cheeger-type inequalities is that they connect an
$L^2$-quantity, the smallest nontrivial eigenvalue of the Laplacian, with an
$L^1$-quantity, the Cheeger constant. Therefore, it seems natural to
interpolate between the exponents 2 and 1. This can be done, as we shall
briefly explain now, but the case $p=1$, which is the case of most interest,
creates additional difficulties. But in fact, for $p=1$, the inequalities that
we are after become equalities, and this conversely is useful for deriving the
inequality for $p=2$.

Thus, similar to the up and down Laplacians on simplicial complexes (see Section \ref{sec:Sim-complex}), 
we shall now introduce the  $p$-Laplace operators on $C^d(\Sigma)$. For $p>1$,
we put $$\alpha_p:(t_1,t_2,\cdots)\mapsto
(|t_1|^{p-2}t_1,|t_2|^{p-2}t_2,\cdots).$$
Since this becomes undetermined for
$p=1$ when $t=0$, we need to modify the definition and let it be set valued,
that is, $$\alpha_1:(t_1,t_2,\cdots)\mapsto
\{(\xi_1,\xi_2,\cdots):\xi_i\in\mathrm{Sgn}(t_i)\},$$
with
 $$\mathrm{Sgn}(t):=\begin{cases}
 \{1\} & \text{if } t>0,\\
 [-1,1] & \text{if }t=0,\\
 \{-1\} & \text{if }t<0,
 \end{cases}
$$
We can then define 
 the $d$-th up $p$-Laplace  operator
 $$L^{up}_{d,p}:=\delta_d^*\alpha_{p}\delta_d,$$
 having for $f\in C^d(\Sigma)$,
 $$L^{up}_{d,p}f=B_{d+1} \alpha_p(B_{d+1}^\top f)$$
 where we identify $\delta_d$ with its standard matrix representation $B_{d+1}^\top $. 
 Analogously, we can also define 
 the $d$-th down $p$-Laplace  operator
 $L^{down}_{d,p}:=\delta_{d-1}\alpha_p\delta_{d-1}^*$,
 having  for $f\in
 C^d(\Sigma)$,
 $L^{down}_{d,p}f=B_d^\top \alpha_p (B_df)$,
 and  the $d$-th
  $p$-Laplace  operator as
 $ L_{d,p}:=L^{up}_{d,p}+ L^{down}_{d,p}$.

The eigenvalue problem of $L^{up}_{d,p}$ is to find  real numbers $\lambda$ and nonzero functions $f:\Sigma_d\to\R$ satisfying 
\[L^{up}_{d,p}f=\lambda\alpha_p(f),\text{ for the case of }p>1,\]
or  
\[0\in L^{up}_{d,1}f-\lambda\alpha_1(f),\text{ for the case of }p=1.\]
In the case of $d=0$, the above nonlinear eigenproblem is actually the spectral problem for the graph $p$-Laplacian \cite{Amghibech,HeinBuhler2009,HeinBuhler2010,CSZ17}. 
Of most interest for us will be  the min-max eigenvalues, that is, those that
can be obtained from Rayleigh quotients as in Lemma \ref{rayleigh}. Thus, we
look for 
\begin{equation}
  \label{plap1}
\lambda_i(L^{up}_{d,p}):=\inf_{\gamma(S)\ge i}\sup_{f\in
  S}\frac{\|B_{d+1}^\top f\|_p^p}{ \|f\|_{p}^p},\;\;i=1,2,\cdots,n,
\end{equation}
where $n=\#\Sigma_d$, and 
\begin{equation*}
\gamma(S):=\begin{cases}
\min\limits\{k\in\mathbb{Z}^+: \exists\; \text{odd continuous map}\; \varphi: S\to \mathbb{S}^{k-1}\} & \text{if}\; S 
\ne\emptyset,\\
0 & \text{if}\; S  
=\emptyset,
\end{cases}
\end{equation*}
denotes the 
Krasnoselskii 
genus of a centrally symmetric set $S\subset\R^n\setminus \{\mathbf{0}\}
$.  As already indicated, the
important case of \eqref{plap1} will be $p=1$.

Obviously, analogous constructions  work  %are possible 
for $L^{down}_{d,p}$.
~\\

{\bf Acknowledgements.} The authors are  grateful to the anonymous referee for comments and suggestions which greatly helped us improve the quality of the presentation of our paper. 
Dong Zhang is supported by grants from the National Natural Science Foundation of China (No. 12401443).

\section{Cheeger-type inequalities on $d$-faces of simplicial complexes}

\subsection{Spectral gap from $d+2$}\label{sec:gap-d+2}
Here, we shall build upon  Sections \ref{sec:Sim-complex} and
\ref{sec:simp-sign}. %{sec:Cheeger-sign}. 
Again, the key is to convert a Cheeger problem for
higher dimensional simplices into one for signed graphs. We thus suggest the following Cheeger-type constants.

As always, we consider a simplicial complex $\Sigma$,  %and, recalling  Definition \ref{def1}, 
  and we denote the collection of its $d$-dimensional 
  simplices by $\Sigma_d$. 
 We shall need a slight modification of the construction in Section
 \ref{sec:simp-sign}. Hereafter, we will   consider the signed graph
 $(\Gamma_d,s)$ on the vertex set $\Sigma_d$, under the up adjacency relation,
 and with the sign function
 \begin{equation}\label{sign}
   s([\tau],[\tau'])=
   \sgn([\tau],\partial [\sigma])\sgn([\tau'],\partial [\sigma])
 \end{equation}
 which is the opposite of the sign function defined in \eqref{asc2}, where $\sigma$ used in \eqref{sign} is the unique $(d+1)$-simplex such that both $\tau$ and $\tau'$ are facets of $\sigma$. 
  
  For disjoint $A,A'\subset \Sigma_d$, let $|E^+(A,A')|=\#\{\{\tau,\tau'\}:\tau\in A,\tau'\in A',s([\tau],[\tau'])=1\}$  and $|E^-(A)|=\#\{\{\tau,\tau'\}:\tau,\tau'\in A,s([\tau],[\tau'])=-1\}$. Let 
$$\beta(A,A')=\frac{2\left(|E^-(A)|+|E^-(A')|+|E^+(A,A')|\right)+|\partial(A\sqcup A')|}{\vol(A\sqcup A')}$$
where $|\partial A|$ is the number of the edges of $(\Gamma_d,s)$ that cross $A$ and $\Sigma_d\setminus A$,   $\vol(A)=\sum_{\tau\in A}\deg \tau$ and $\deg \tau=\#\{\sigma\in \Sigma_{d+1}:\tau\subset \sigma\}$. 

Then we introduce the $k$-th  Cheeger constant   on $\Sigma_d$:
$$h_k(\Sigma_d)=\min\limits_{\text{disjoint } A_1,A_2,\ldots,A_{2k-1},A_{2k}\text{ in }\Sigma_d}\max\limits_{1\le i\le k}\beta(A_{2i-1},A_{2i}).$$ 
 $h_k(\Sigma_d)=0$ if and only if $(\Gamma_d,s)$ has exactly $k$ balanced components.

\begin{remark}
 For $d=0$,  the constant $h_k(\Sigma_0)$ reduces to the $k$-way Cheeger
 constant %\ref{cg1} 
 of a graph  \cite{LGT12}.
\end{remark}

\begin{theorem}\label{thm:anti-signed-Cheeger}
For any simplicial complex and every $d\ge 0$,
\begin{equation}\label{eq:Cheeger-1-complex}
\frac{ h_1(\Sigma_d)^2}{2(d+1)}\le d+2-\lambda_n(\Delta^{up}_d)\le 2h_1(\Sigma_d),
\end{equation}
where $n=\#\Sigma_d$. Moreover, there exists an absolute constant $C$  such that for any simplicial complex, and for any $k\ge 1$, 
\begin{equation}\label{eq:Cheeger-k-complex}
\frac{ h_k(\Sigma_d)^2}{Ck^6(d+1)}\le d+2-\lambda_{n+1-k}(\Delta^{up}_d)\le 2h_k(\Sigma_d).
\end{equation}

\end{theorem}

\begin{proof}
We first show
\begin{equation}
  \label{r2}
  d+2-\lambda_{n-i+1}(\Delta^{up}_d)=(d+1)\lambda_i(\Delta_{(\Gamma_d,s)}),\;\;\;i=1,\ldots,n.
  \end{equation}
We have
\begin{eqnarray*}
 &&(d+2)\sum_{\tau\in \Sigma_d}\deg \tau\cdot f(\tau)^2-\sum\limits_{\sigma\in
  \Sigma_{d+1}}\left(\sum_{\tau\in
    \Sigma_d,\tau\subset\sigma}\sgn([\tau],\partial[\sigma])f(\tau)\right)^2\\
  &=&\sum_{[\tau]\sim[\tau']}\left(f(\tau)-\sgn([\tau],\partial[\sigma])\sgn([\tau'],\partial[\sigma])f(\tau'))\right)^2.
  \end{eqnarray*}
Recalling \eqref{sign}, this yields  the identity 
$$d+2-\frac{\sum\limits_{\sigma\in \Sigma_{d+1}}\left(\sum_{\tau\in
      \Sigma_d,\tau\subset\sigma}\sgn([\tau],\partial[\sigma])f(\tau)\right)^2}{\sum_{\tau\in
    \Sigma_d}\deg \tau\cdot
  f(\tau)^2}=(d+1)\frac{\sum_{[\tau]\sim[\tau']}\left(f(\tau)-s(\tau,\tau')f(\tau'))\right)^2}{\sum_{\tau\in
    \Sigma_d}\widetilde{\deg}\,\tau f(\tau)^2}$$
for  the  Rayleigh quotients, 
where $[\tau]\sim[\tau']$ represents an  edge in the underlying  signed graph
${(\Gamma_d,s)}$, and $\widetilde{\deg}\,\tau:=(d+1)\deg \tau$ is the degree of
$\tau$ in  ${(\Gamma_d,s)}$. (Whenever $\tau \subset \sigma \in \Sigma_{d+1}$,
this connects $\tau$ with $d+1$ other $d$-simplices.) Recalling Lemma
\ref{rayleigh}, this shows \eqref{r2}. 

Moreover, since  $\frac{1}{d+1}h_k(\Sigma_d)$ is the $k$-th Cheeger constant of the signed graph $(\Gamma_d,s)$, by %the Cheeger inequality and the higher order Cheeger inequalities 
the Cheeger inequality \eqref{atay} for signed graphs, % Theorem \ref{thm:Cheeger-sign},
we have  $$\frac{\lambda_1(\Delta_{(\Gamma_d,s)})}{2}\le \frac{h_1(\Sigma_d)}{d+1}\le \sqrt{2\lambda_1(\Delta_{(\Gamma_d,s)})}.$$ And by Theorem \ref{thm:signed-Cheeger},  there exists  an absolute constant $C$  such that for any signed graph and any $k\ge 1$, $$\frac{\lambda_k(\Delta_{(\Gamma_d,s)})}{2}\le \frac{h_k(\Sigma_d)}{d+1}\le Ck^3 \sqrt{\lambda_k(\Delta_{(\Gamma_d,s)})}.$$ In consequence, we obtain
$$ \frac{d+2-\lambda_{n}(\Delta^{up}_d)}{2}\le h_1(\Sigma_d)\le \sqrt{2(d+1)(d+2-\lambda_{n}(\Delta^{up}_d))}$$
and
$$\frac{d+2-\lambda_{n+1-k}(\Delta^{up}_d)}{2}\le h_k(\Sigma_d)\le Ck^3\sqrt{(d+1)(d+2-\lambda_{n+1-k}(\Delta^{up}_d))} .$$
Then, we have verified \eqref{eq:Cheeger-1-complex} and \eqref{eq:Cheeger-k-complex}. 
\end{proof}
By Theorem \ref{thm:anti-signed-Cheeger}, $\lambda_n(\Delta^{up}_d)=d+2$ if
and only if  $h_1(\Sigma_d)=0$,  if and only if the associated  signed graph
${(\Gamma_d,s)}$ has a  balanced component. The latter fact follows from
Proposition \ref{prop:q+2 eigen}, remembering that the sign we are currently using
is the opposite of the one in that proposition.

\subsection{Spectral gap from 0}
\label{sec:gap-0}
Theorem \ref{thm:anti-signed-Cheeger} of  the previous section contains the estimates for the spectral gap from %with 
$d+2$. %which should be seen as the other end of the spectrum of $\Delta^{up}_d$. 
However, the more  important  estimate is the one for the spectral gap from %with 
$0$, namely, the Cheeger-type estimate for the first non-trivial eigenvalue of
the Eckmann Laplacian. 
For that purpose, we shall now introduce a new  Cheeger  constant. The key point is that we consider generalized (i.e., with both positive and negative multiplicities) multisets
  of $d$-simplices, in order to be able to take account of (positive or
  negative) multiplicities, as these also enter into the coboundary relations
  and therefore implicitly into the
  eigenvalues.
  
\begin{itemize}
\item[(D1)] A \emph{(generalized) multiset} is a  pair $(S,m)$, where $S$ is the underlying set of the
  multiset, formed from its distinct elements, and $m:S\to\mathbb{Z}$ is an
  integer-valued function, giving the \emph{multiplicity}. We point out that
  this multiplicty is allowed to also take negative values, in order to
  account for orientations. For convenience, we usually write $S$ instead of
  $(S,m)$, and simply speak of a multiset,  and we use $|S|
:=%\coloneqq 
\sum_{s\in S}|m(s)|$ to indicate the \emph{size} of the multiset $S$. 

As the underlying set, we take  $\Sigma_d$. We write $S\subset_M \Sigma_d$
when  $S$ is a multiset on  the underlying set $\Sigma_d$ with
multiplicities in $\{-M,\ldots,0,\ldots,M\}$. The coboundary  $\partial^*_{d+1}S$ of such a multiset $S$ is defined as the multiset of  all $(d+1)$-simplices that have a member of $S$ in its boundary, together with the appropriate multiplicities. Thus, 
  each $\sigma\in \Sigma_{d+1}$ has the multiplicity $\sum_{\tau\in
    \Sigma_d}m(\tau)\mathrm{sgn}([\tau],\partial[\sigma])$, where $m(\tau)$ is
  the multiplicity of $\tau$ in $S$. And the support of $\partial^*_{d+1}S$
  then consists of all such simplices with non-zero multiplicity. We  define
  $\vol(S):=\sum_{\tau\in \Sigma_d}\deg \tau\cdot |m(\tau)|$  as the volume  of the multiset $S$.  %$\partial^*_{d+1}(\tau)=\sum_{\sigma\in \Sigma_{d+1}}\mathrm{sgn}(\tau,\partial\sigma)\sigma$
\begin{defn} For  $d\ge 0$, 
\begin{equation}\label{eq:combinatorial-h(S_d)}h(\Sigma_d)=\min\limits_{\substack{S\subset_M \Sigma_d \\ S\neq \partial^*_{d}(T),\forall T\subset_M \Sigma_{d-1}}}\frac{|\partial^*_{d+1} S|}{\min\limits_{S'\ne\emptyset:\partial^*_{d+1}S'=\partial^*_{d+1}S}\vol(S')}
\end{equation}
is constant when $M$ is sufficiently large. And for such a large number $M$, we call $h(\Sigma_d)$ the \emph{Cheeger constant} on $\Sigma_d$.
\end{defn}
\item[(D2)]  We shall now give several different definitions of
    $h(\Sigma_d)$, and then  show that these definitions all
    agree. First, we describe the Cheeger constant  as the
  $\mathbb{Z}$-expander:
  \begin{defn}
Let $$h(\Sigma_d)=\min\limits_{\phi\in
  C^d(\Sigma,\mathbb{Z})\setminus\mathrm{Im\,}\delta}\frac{\|\delta\phi\|_1}{\min\limits_{\psi\in\mathrm{Im\,}\delta}\|\phi+\psi\|_{1,\deg}}.$$
\end{defn}
We point out that in contrast to the
definition \eqref{skm} of $h^d(\Sigma)$, here  we use  $\mathbb{Z}$- instead of   $\mathbb{Z}_2$-coefficients, and we
use\ the  (weighted) $l^1$-norm,   where $\|\phi \|_{1,\deg}:=\sum_{\tau\in
  \Sigma_d}\deg \tau\cdot |\phi(\tau)|$, instead of the Hamming norm.

\item[(D3)]  %Using the nonlinear eigenvalue problem, we 
Anticipating Section \ref{sec:p-lap}, and similar to the graph 1-Laplacian, we define the up 1-Laplacian eigenvalue problem on $\Sigma_d$ as the nonlinear eigenvalue problem 
\begin{equation}\label{eq:1-lap}
 0\in\nabla \|B_{d+1}^\top\mathbf{x}\|_1-\lambda\nabla\|\mathbf{x}\|_{1,\deg}   
\end{equation}
where $\nabla $ represents the usual subgradient \cite{Clarke}. Here, given a convex function $F$ on a Hilbert space, the subgradient of $F$ at $\mathbf{x}$, denoted by $\nabla F(\mathbf{x})$,  is defined as $\nabla F(\mathbf{x})=\{\mathbf{v}\;:F(\mathbf{z})-F(\mathbf{x})\ge \langle \mathbf{v},\mathbf{z}-\mathbf{x}\rangle\}$.

We let $\lambda_{I_d}(\Delta^{up}_{d,1})$ be the smallest non-trivial eigenvalue of the up 1-Laplacian, where  $I_d:=\dim \mathrm{Image}(B_d^\top)+1=\mathrm{rank}(B_d)+1$. 
To describe  $\lambda_{I_d}(\Delta^{up}_{d,1})$, we first introduce orthogonality w.r.t.\ a given norm. 
        For a norm $\|\cdot\|$ on a real linear space with an inner product $\langle \cdot ,\cdot \rangle$,  we say that  $\mathbf{x}$ is \emph{$\|\cdot\|$-orthogonal} to $\mathbf{y}$ if there exists $\mathbf{u}\in \nabla\|\mathbf{x}\|$ satisfying $\langle \mathbf{u},\mathbf{y}\rangle=0$. We say $\mathbf{x}$ is $\|\cdot\|$-orthogonal to a non-empty set $Y$ if $\mathbf{x}$ is $\|\cdot\|$-orthogonal to all    $\mathbf{y}\in Y$. Clearly, if $\|\cdot\|=\|\cdot\|_2$ is the standard $l^2$-norm, then the $\|\cdot\|_2$-orthogonality  reduces to the usual orthogonality w.r.t.\ the standard inner product. 
\begin{defn}Let
$$h(\Sigma_d)=\lambda_{I_d}(\Delta^{up}_{d,1})=
\min\limits_{\mathbf{x}\bot^1 \mathrm{Image}(B_d^\top)}\frac{\|B_{d+1}^\top\mathbf{x}\|_1}{\|\mathbf{x}\|_{1,\deg}}$$
%the Cheeger constant on $\Sigma_d$, 
where $\mathbf{x}\bot^1\mathrm{Image}(B_d^\top)$ indicates that $\mathbf{x}$ is $\|\cdot\|_{1,\deg}$-orthogonal to $\mathrm{Image}(B_d^\top)$, i.e.\ $\mathbf{u}\in\mathrm{Image}(B_d^\top)^\bot$ for some $ \mathbf{u}\in\nabla \|\mathbf{x}\|_{1,\deg}$.
\end{defn}

\item[(D4)] The norm $\|\cdot\|_{1,\deg}$ on $C^{d}(\Sigma)$ induces a quotient norm on
$C^{d}(\Sigma)/\mathrm{image}(\delta_{d-1})$, which will be  denoted by $\|\cdot\|$
for simplicity.  More precisely, for any equivalence class $[\mathbf{x}]\in C^{d}(\Sigma)/\mathrm{image}(\delta_{d-1})$,  let $\| [\mathbf{x}]\|=\inf\limits_{x'\in [x]}\|\mathbf{x}'\|_{1,\deg}$. Then 
$$h(\Sigma_d)=
\min\limits_{0\ne [\mathbf{x}]\in C^{d}(\Sigma)/\mathrm{image}(\delta_{d-1})}\frac{\|\delta_d\mathbf{x}\|_1}{\| [\mathbf{x}]\|}=\min\limits_{0\ne [\mathbf{x}]\in C^{d}(\Sigma,\mathbb{Z})/\mathrm{image}(\delta_{d-1})}\frac{\|\delta_d\mathbf{x}\|_1}{\| [\mathbf{x}]\|}.$$
%and it is interesting that 
\begin{defn}
In the case of $\tilde{H}^{d}(\Sigma,\R)= 0$, let 
$$h(\Sigma_d)=\min\limits_{\mathbf{y}\in \mathrm{image}(\delta_{d})\setminus\{\mathbf{0}\}}\frac{\|\mathbf{y}\|_1}{\|\mathbf{y}\|_{\mathrm{fil}}}=\frac{1}{\max\limits_{\mathbf{y}\in \mathrm{image}(\delta_{d})\setminus\{\mathbf{0}\}}\|\mathbf{y}\|_{\mathrm{fil}}/\|\mathbf{y}\|_1}=\frac{1}{\|\delta_d^{-1}\|_{\mathrm{fil}}}$$
where $\|\mathbf{y}\|_{\mathrm{fil}}:=%\coloneqq
\inf\limits_{x\in\delta_d^{-1} (\mathbf{y})}\|\mathbf{x}\|_{1,\deg}$ is the filling norm of $\mathbf{y}$, and\\ 
$\|\delta_d^{-1}\|_{\mathrm{fil}}:=\max\limits_{\mathbf{y}\in \mathrm{image}(\delta_{d})\setminus\{\mathbf{0}\}}\|\mathbf{y}\|_{\mathrm{fil}}/\|\mathbf{y}\|_1$ is called the  filling profile by Gromov (see Section 2.3 in \cite{Gromov}).  
\end{defn}
\end{itemize}

\begin{theorem}
The four definitions in (D1)--(D4)  are equivalent.
\end{theorem}
 
\begin{proof}
We start with (D3). Since $\mathrm{Image}(B_d^\top)\subset\mathrm{Ker}(B_{d+1}^\top)$, by Theorem 2.1 in \cite{Jost/Zhang21c},
\begin{align}
\lambda_{I_d}(\Delta^{up}_{d,1})&=\inf\limits_{\mathbf{x}\in \R^n\setminus \mathrm{Image}(B_d^\top)}\frac{\|B_{d+1}^\top\mathbf{x}\|_1}{\inf\limits_{\mathbf{z}\in \mathrm{Image}(B_d^\top)}\|\mathbf{x}+\mathbf{z}\|_{1,\deg}}\label{eq:B-inf}
\\&=\inf\limits_{[\mathbf{x}]\in \R^n/ \mathrm{Image}(B_d^\top)}\frac{\|B_{d+1}^\top\mathbf{x}\|_1}{\| [\mathbf{x}]\|}\label{eq:B-quotient}
\\&=\inf\limits_{\mathbf{x}\in \R^n:\nabla \|\mathbf{x}\|_{1,\deg}\bigcap\mathrm{Image}(B_d^\top)^\bot\ne\emptyset}\frac{\|B_{d+1}^\top\mathbf{x}\|_1}{\|\mathbf{x}\|_{1,\deg}} \label{eq:B-constraint} 
\end{align}
where $n=\#\Sigma_d$, 
$$[\mathbf{x}]=\left\{\mathbf{y}\in\R^n:\mathbf{y}-\mathbf{x}\in\mathrm{Image}(B_d^\top)\right\}$$
and $$\| [\mathbf{x}]\|=\inf\limits_{\mathbf{x}'\in  [\mathbf{x}]}\|\mathbf{x}+\mathbf{z}\|_{1,\deg}.$$ 
In fact, the definition of the norm $\|\cdot\|$ on the quotient space $\R^n/ \mathrm{Image}(B_d^\top)$ implies  $$\|[\mathbf{x}]\|=\inf\limits_{\mathbf{z}\in \mathrm{Image}(B_d^\top)}\|\mathbf{x}+\mathbf{z}\|_{1,\deg}.$$ 
Moreover,  Proposition 2.3 in \cite{Jost/Zhang21c} yields that $\|[\mathbf{x}]\|=\|\mathbf{x}\|_{1,\deg}$ if and only if $\mathbf{x}$ satisfies $\nabla \|\mathbf{x}\|_{1,\deg}\bigcap\mathrm{Image}(B_d^\top)^\bot\ne\emptyset$, 
that is, the minimization problem  \[\inf\limits_{\mathbf{x}'\in\R^n:\mathbf{x}'-\mathbf{x}\in \mathrm{Image}(B_d^\top)}\|\mathbf{x}'\|_{1,\deg}\] reaches its minimum at some points  in the set  $\{\mathbf{x}\in \R^n:\nabla \|\mathbf{x}\|_{1,\deg}\bigcap\mathrm{Image}(B_d^\top)^\bot\ne\emptyset\}$. 
So, the above three quantities \eqref{eq:B-inf}, \eqref{eq:B-quotient} and \eqref{eq:B-constraint}  coincide. 
Using the $l^1$-type  orthogonal notation $\bot^1$, since  $\mathbf{x}\bot^1\mathrm{Image}(B_d^\top)$ means that  $\mathbf{u}\bot \mathrm{Image}(B_d^\top)$  for some  $  \mathbf{u}\in\nabla \|\mathbf{x}\|_{1,\deg}$, the constraint $\{\mathbf{x}\in \R^n:\nabla \|\mathbf{x}\|_{1,\deg}\bigcap\mathrm{Image}(B_d^\top)^\bot\ne\emptyset\}$ in  \eqref{eq:B-constraint} can be reduced to $\{\mathbf{x}\in \R^n:\mathbf{x}\bot^1\mathrm{Image}(B_d^\top)\}$ as shown in (D3).

Similar to the proof of Proposition 3.7 in \cite{Jost/Zhang21b}, %Proposition \ref{pro:Lovasz-eigen-pre},
we can apply Theorem 2.4 in  \cite{Jost/Zhang21c} %Lemma \ref{lemlov1} 
to derive that every eigenvalue of the up 1-Laplacian eigenproblem \eqref{eq:1-lap}  has an eigenvector in the set of the extreme points associated with   the function pair  $(\|B_{d+1}^\top\cdot\|_1,\|\cdot\|_{1,\deg})$ since both $\|B_{d+1}^\top\cdot\|_1$ and $\|\cdot\|_{1,\deg}$ are piecewise linear. 
We shall now describe   these extreme points in more detail.

The unit $l^1$-sphere  $\{\mathbf{x}\in\R^n:\|\mathbf{x}\|_{1,\deg}=1\}$ can be represented as a union $P_1\cup\cdots\cup P_k$ of finitely many convex polytopes of dimension $(n-1)$ on which  both $\|B_{d+1}^\top\cdot\|_1$ and $\|\cdot\|_{1,\deg}$ are linear, and we let $k$ here be the smallest such integer. 
%The unit $l^1$-sphere  $\{\mathbf{x}\in\R^n:\|\mathbf{x}\|_{1,\deg}=1\}$ can be represented as a union of finitely many convex polytopes of dimension $(n-1)$ such that both $\|B_{d+1}^\top\cdot\|_1$ and $\|\cdot\|_{1,\deg}$ are linear on each convex polytope. 
%Let $k$ be the smallest  possible number of  such convex polytopes, and let $\{P_1,\cdots,P_k\}$ be the family of convex polytopes of dimension $(n-1)$, i.e.\ $\{\mathbf{x}\in\R^n:\|\mathbf{x}\|_{1,\deg}=1\}=P_1\cup\cdots\cup P_k$ and $\|B_{d+1}^\top\cdot\|_1$ and $\|\cdot\|_{1,\deg}$ are linear when  restricted on  $P_i$ for any $i$. 
%Denote by $\mathrm{Ext}(\|B_{d+1}^\top\cdot\|_1,\|\cdot\|_{1,\deg})$  the union of the vertices of $P_i$ for all $i$. 
For $i=1,\dots ,k$, let  $\mathrm{Ext}(\|B_{d+1}^\top\cdot\|_1,\|\cdot\|_{1,\deg})$  be the vertex set  of $P_i$. 
Clearly, $\mathrm{Ext}(\|B_{d+1}^\top\cdot\|_1,\|\cdot\|_{1,\deg})$ is a finite set, and its elements are called the extreme points determined by the  function pair $(\|B_{d+1}^\top\cdot\|_1,\|\cdot\|_{1,\deg})$. %are referred to the following sense. 

Since all the entries of the  matrix $B_{d+1}^\top$ and the degrees are rational numbers,  by the theory of systems of linear equations, 
$\mathrm{Ext}(\|B_{d+1}^\top\cdot\|_1,\|\cdot\|_{1,\deg})\subset \mathbb{Q}^n$. 
Let $M$ be a sufficiently  large natural number that is greater than the least  common multiple of all the denominators of the  components of all points in $\mathrm{Ext}(\|B_{d+1}^\top\cdot\|_1,\|\cdot\|_{1,\deg})$. 
Then, %all the extreme points belong to 
$\mathrm{Ext}(\|B_{d+1}^\top\cdot\|_1,\|\cdot\|_{1,\deg})\subset \left\{t\mathbf{x}:t\ge0\text{ and }\mathbf{x}\in \{-M,\ldots,-1,0,1,\ldots,M\}^{n}\right\}$, and thus every eigenvalue has an eigenvector in $\left\{t\mathbf{x}:t\ge0\text{ and }\mathbf{x}\in \{-M,\ldots,-1,0,1,\ldots,M\}^{n}\right\}$.  %for some positive integer $m$. 
Since both $\|B_{d+1}^\top\cdot\|_1$ and $\|\cdot\|_{1,\deg}$ are positively one-homogeneous,  we further derive that  every eigenvalue has an eigenvector in the set $\{-M,\ldots,-1,0,1,\ldots,M\}^{n}$. 
Moreover, the minimizations   \eqref{eq:B-inf} and \eqref{eq:B-constraint} can reach their minima at some points in $\{-M,\ldots,-1,0,1,\ldots,M\}^{n}$, and the minimization problem \eqref{eq:B-quotient}  achieves its minima at some equivalence  class $[\mathbf{x}]$ for  some  $\mathbf{x}\in \{-M,\ldots,-1,0,1,\ldots,M\}^{n}$. 
That means, we can use $\{-M,\ldots,-1,0,1,\ldots,M\}^n$ instead of $\R^n$ in the constraints of these three minimization problems \eqref{eq:B-inf}, \eqref{eq:B-quotient} and \eqref{eq:B-constraint}. 
It follows from  $\{-M,\ldots,-1,0,1,\ldots,M\}^n\subset \mathbb{Z}^n\subset \R^n$ that one can also replace $\R^n$ by $\mathbb{Z}^n$ in  the constraints of these three minimization problems \eqref{eq:B-inf}, \eqref{eq:B-quotient} and \eqref{eq:B-constraint}. 

We now proceed to prove the equivalence of (D1)--(D4). 

By using $\mathbb{Z}^n$ instead of $\R^n$ in \eqref{eq:B-inf} and equivalently replacing $\mathbb{Z}^n$ by $C^d(\Sigma,\mathbb{Z})$, and $\mathrm{Image}(B_d^\top)$ by $\mathrm{Im\,}\delta$, we obtain that (D2) is a reformulation of \eqref{eq:B-inf}. 
Similarly, (D4) is a reformulation of  \eqref{eq:B-quotient}. And, if $\tilde{H}^{d}(\Sigma,\R)= 0$, then $\mathrm{image}(\delta_{d-1})=\mathrm{ker}(\delta_{d})$, which implies  $C^{d}(\Sigma)/\mathrm{image}(\delta_{d-1})=C^{d}(\Sigma)/\mathrm{ker}(\delta_{d})\cong \mathrm{image}(\delta_{d})$ and  $\delta_d^{-1}(\mathbf{y})=[\mathbf{x}]$ for any $\mathbf{y}\in \mathrm{image}(\delta_{d})$. 
Note that the filling norm  $$\|\mathbf{y}\|_{\mathrm{fil}}\coloneqq \inf_{x\in\delta_d^{-1} (\mathbf{y})}\|\mathbf{x}\|_{1,\deg}$$ coincides with $\|[\mathbf{x}]\|$, and $\|\mathbf{y}\|_1=\|\delta_d\mathbf{x}\|_1=\|B_{d+1}^\top\mathbf{x}\|_1$. Therefore, (D4) and \eqref{eq:B-quotient} coincide. 

By using  $\{-M,\ldots,-1,0,1,\ldots,M\}^n$ instead of $\R^n$ in \eqref{eq:B-inf}, we can similarly identify every generalized multiset  $S\subset_M\Sigma_d$ with a unique  $\mathbf{x}\in \{-M,\ldots,-1,0,1,\ldots,M\}^n$ by identifying  $x_\tau$ with $m(\tau)$ for any $\tau\in\Sigma_d$, where $m(\tau)$ is the generalized multiplicity of $\tau$ in $S$. 
For such $S$ and $\mathbf{x}$, we then have $\vol(S)=\|\mathbf{x}\|_{1,\deg}$ and $|\partial^*_{d+1} S|=\|B_{d+1}^\top\mathbf{x}\|_1$. 
If $\tilde{H}^{d}(\Sigma,\R)\ne 0$,  then $\mathrm{image}(B_{d}^\top)$ is a proper subset of $ \mathrm{ker}(B_{d+1}^\top)$, and thus \eqref{eq:B-inf} is zero. In this case, there exists  $S'\ne\emptyset$ such that  $\partial^*_{d+1}S'=\partial^*_{d+1}S=\emptyset$, which means that 
\eqref{eq:combinatorial-h(S_d)} also equals zero. 
If $\tilde{H}^{d}(\Sigma,\R)= 0$,  then $\mathrm{image}(B_{d}^\top)= \mathrm{ker}(B_{d+1}^\top)$,  and thus for such $S$ and $\mathbf{x}$ with $\mathbf{x}\not\in\mathrm{ker}(B_{d+1}^\top)$, we have\[\inf\limits_{\mathbf{z}\in \mathrm{image}(B_d^\top)}\|\mathbf{x}+\mathbf{z}\|_{1,\deg}=\inf\limits_{\mathbf{x}'\in\R^n:\mathbf{x}'-\mathbf{x}\in \mathrm{ker}(B_{d+1}^\top)}\|\mathbf{x}'\|_{1,\deg}=\min\limits_{S'\ne\emptyset:\partial^*_{d+1}S'=\partial^*_{d+1}S}\vol(S').\] 
Therefore, \eqref{eq:combinatorial-h(S_d)} and \eqref{eq:B-inf} actually represent  the same quantity which has been  denoted by $h(\Sigma_d)$. 
This proves the equivalence between (D1), (D2), (D3) and (D4).
\end{proof}

% This means that $\lambda_{I_d}(\Delta^{up}_{d,1})$ can be expressed as a combinatorial optimization, or equivalently, an integer  programming with constraint on $\{-N,\ldots,-1,0,1,\ldots,N\}^n$, and thus we would like to call

Hence, the four equivalent definitions in (D1)--(D4) represent the same Cheeger constant $h(\Sigma_d)$ from different viewpoints:

(D1) provides a  combinatorial formulation of the Cheeger constant $h(\Sigma_d)$ using the language of multi-sets  in combinatorics, which means that our Cheeger constant is actually a combinatorial quantity.  %\commentred{I think we need to find more explicit combinatorial explanations and geometric explanations for $h(\Sigma_d)$.}  

(D2) presents  $h(\Sigma_d)$ as a $\mathbb{Z}$-expander, and makes it clear that 
$$h(\Sigma_d)=0 \Longleftrightarrow \tilde{H}^d(\Sigma,\R)\ne0,\;\; \forall d\ge 0.$$

As we have discussed, the Cheeger constant defined as a $\mathbb{Z}_2$-expander  violates the Cheeger inequality on simplicial complexes. However, as a  $\mathbb{Z}$-expander, it is possible to get a Cheeger inequality. 

(D3) shows that $h(\Sigma_d)$ coincides with the smallest nontrivial 1-Laplacian eigenvalue, which was already known in both the graph and the domain settings.

(D4) reveals the non-obvious fact that  $h(\Sigma_d)$ has a deep relation with Gromov's filling profile. This is an equivalent reformulation of \eqref{eq:combinatorial-h(S_d)} using the language of norms on cochain groups, which helps us to further  understand the formula \eqref{eq:combinatorial-h(S_d)}.  

%It should be noted that $\tilde{H}^{d}(\Sigma,\R)\ne 0$ if and only if  $h(\Sigma_d)=0$.  More precisely, %if  $\tilde{H}^{d}(\Sigma,\R)= 0$, then 
%according to Theorems \ref{thm:smallest-non-zero} and \ref{piecewise-linear-vertex}, as well as the results in Section \ref{sec:structure-eigenspace}, 
In addition, for sufficiently  large numbers $M\in\mathbb{Z}_+$,
$$
h(\Sigma_d)\xlongequal[]{\text{if } \tilde{H}^{d}(\Sigma,\R)= 0}\min\limits_{\substack{S\subset_M \Sigma_d \\ \partial^*_{d+1}S\ne \emptyset}}\frac{|\partial^*_{d+1} S|}{\min\limits_{S':\partial^*_{d+1}S'=\partial^*_{d+1}S}\vol(S')}>0.   
$$
%where    $\vol[ S]\coloneqq \min\limits_{S':\partial^*_{d+1}S'=\partial^*_{d+1}S}\vol(S')$. 

For the case of $d=0$, we can take $M=1$, and then $h(\Sigma_0)$ reduces to the usual Cheeger constant on graphs. %The following  preliminary result indicates that such a   constant $h(\Sigma_d)$ is  a good candidate for  Cheeger-type  inequalities. 
 The following result shows that the constant $h(\Sigma_d)$ satisfies Cheeger-type inequalities, and therefore provides a  solution to the problem formulated in the introduction.

\begin{pro}\label{pro:rough-Cheeger}
Suppose that $\deg \tau>0$, $\forall \tau\in \Sigma_d$. Then, $$\frac{h^2(\Sigma_d)}{|\Sigma_{d+1}|}\le \lambda_{I_d}(\Delta_d^{up})\le \vol(\Sigma_d)h(\Sigma_d).$$
\end{pro}
\begin{proof}
For simplicity, we let  $h=h(\Sigma_d)$ and $\lambda=\lambda_{I_d}(\Delta_d^{up})$. 
%We shall prove $$\frac{\min\limits_{\tau\in \Sigma_{d}}\deg \tau}{\#\Sigma_{d+1}}h^2\le \lambda\le \vol(\Sigma_d)h^2.$$ %The  proof of Proposition \ref{pro:rough-Cheeger} is then  completed by noting that $h\le 1\le \deg \tau$,  $\forall \tau\in \Sigma_d$. 
%Let  $I_d=\mathrm{rank}(B_{d})$. 
Note that  $\lambda$ and $h$ are the $I_d$-th min-max eigenvalues of the $d$-th normalized up Laplacian and  the $d$-th up 1-Laplacian, respectively. %We only need to prove that, for any $k\ge 1$, 
%$$\sqrt{\frac{1}{\sum\limits_{\tau\in \Sigma_d}\deg \tau\cdot} \lambda_k}\le h_k\le \sqrt{\frac{\#\Sigma_{d+1}}{\min\limits_{\tau\in \Sigma_{d}}\deg \tau\cdot} \lambda_k}.$$
By the following  elementary inequalities
$$\min\limits_\tau\deg \tau  \le \frac{\|\mathbf{x}\|_{1,\deg}^2}{\|\mathbf{x}\|_{2,\deg}^2}\le \sum_{\tau\in \Sigma_d}\deg \tau  \; \text{ 
and }\; 1\le \frac{\|B_{d+1}^\top\mathbf{x}\|_1^2}{\|B_{d+1}^\top\mathbf{x}\|_2^2}\le \#\Sigma_{d+1},$$
we have \begin{equation}\label{eq:inequal-for-min/max}
\frac{1}{\sum_{\tau\in \Sigma_d}\deg \tau}\frac{\|B_{d+1}^\top\mathbf{x}\|_2^2}{\|\mathbf{x}\|_{2,\deg}^2} \le\frac{\|B_{d+1}^\top\mathbf{x}\|_1^2}{\|\mathbf{x}\|_{1,\deg}^2}\le   \frac{\#\Sigma_{d+1}}{\min\limits_\tau\deg \tau}\frac{\|B_{d+1}^\top\mathbf{x}\|_2^2}{\|\mathbf{x}\|_{2,\deg}^2}.
\end{equation}

Recalling the min-max eigenvalues \eqref{plap1} and \eqref{plap1+}, since the $k$-th min-max eigenvalue of the $d$-th up $1$-Laplacian $\Delta^{up}_{d,1}$ is 
$$\lambda_{k}(\Delta^{up}_{d,1})=\inf_{\gamma(S)\ge k}\sup\limits_{f\in S\setminus \{\mathbf{0}\}}\frac{\|B_{d+1}^\top\mathbf{x}\|_1}{\|\mathbf{x}\|_{1,\deg}},$$
while the $k$-th smallest eigenvalue of  $d$-th normalized up Laplacian is
$$\lambda_{k}(\Delta^{up}_d)=\inf_{\gamma(S)\ge k}\sup\limits_{f\in S\setminus \{\mathbf{0}\}}\frac{\|B_{d+1}^\top\mathbf{x}\|_2^2}{\|\mathbf{x}\|_{2,\deg}^2},$$
we can apply \eqref{eq:inequal-for-min/max} to derive
\begin{equation}\label{eq:m-lambda-k-up}
\frac{1}{\sum\limits_{\tau\in \Sigma_d}\deg\tau}\lambda_{k}(\Delta^{up}_{d})\le \left(\lambda_{k}(\Delta^{up}_{d,1})\right)^2 \le \frac{\#\Sigma_{d+1}}{\min\limits_\tau\deg\tau}\lambda_{k}(\Delta^{up}_{d}). 
\end{equation} 
This is a general inequality relating $\lambda_{k}(\Delta^{up}_{d,1})$ and $\lambda_{k}(\Delta^{up}_d)$. By taking $k=I_d$ in \eqref{eq:m-lambda-k-up}, and by observing that $h=h(\Sigma_d)=\lambda_{I_d}(\Delta^{up}_{d,1})$, we finally obtain
\begin{align*}
    \frac{1}{\vol(\Sigma_d)}\lambda=\frac{1}{\sum\limits_{\tau\in \Sigma_d}\deg\tau}\lambda\le h^2 \le \frac{\#\Sigma_{d+1}}{\min\limits_\tau\deg\tau}\lambda.
    \end{align*}
The proof %of $\frac{h^2(\Sigma_d)}{\#\Sigma_{d+1}}\le \lambda_{I_d}(\Delta_d^{up})\le \vol(\Sigma_d)h(\Sigma_d)$ 
is then  completed by noting that $h\le 1\le \deg \tau$,  $\forall \tau\in \Sigma_d$.  
\end{proof}

\begin{remark}\label{remark:down-Cheeger}
 We can also define the down Cheeger constant {(for $d\ge1$)}
$$h_{down}(\Sigma_d)\coloneqq \min\limits_{x\bot^1 \mathrm{Image}(B_{d+1})}\frac{\|B_{d}\mathbf{x}\|_1}{\|\mathbf{x}\|_{1,\deg}}=\lambda_{I_{d+1}}(\Delta^{down}_{d,1})$$
which  possesses a combinatorial  reformulation that is similar to  \eqref{eq:combinatorial-h(S_d)},  where   $I_{d+1}:=\dim \mathrm{Image}(B_{d+1})+1=\mathrm{rank}(B_{d+1})+1$. 

Consider a $d$-dimensional  combinatorial manifold $\Sigma$, that is, a $d$-dimensional topological manifold that has a triangulation as a simplicial complex. %possessing  a simplicial complex structure. 
As a manifold, we assume that $\Sigma$ is connected and  has no  boundary. Then, $B_{d+1}$ is a $|\Sigma_d|\times 1$ matrix of rank $1$, and  $I_{d+1}=\dim\mathrm{Image}(B_{d+1})+1=\mathrm{rank}(B_{d+1})+1=1+1=2$. Therefore, in particular, $\lambda_{I_{d+1}}=\lambda_2$. Moreover,   the down adjacency relation  induces a graph on $\Sigma_d$, and we can infer the following Cheeger inequality: %Poincare duality   
$$\frac{h^2_{down}(\Sigma_d)}{2}\le \lambda_{2}(\Delta_d^{down}) %=\frac{2}{d+2}\lambda_{I_d}(\Delta_{d-1}^{up})
\le 2h_{down}(\Sigma_d).$$
In fact, Theorem 2.7 in \cite{Steenbergen14}   closely resembles the above
inequality, and the assumption made there for the lower bound that every $(d-1)$-dimensional
simplex is incident to at most two $d$-simplices is satisfied for a
combinatorial manifold. 
\end{remark}

In the sequel, $M$ will be used to denote a manifold. 

\begin{defn}
Let $M$ be a $d$-dimensional orientable compact closed Riemannian manifold, and let $c>1$. A triangulation $T$  of $M$ is \emph{$c$-uniform} if %there exists $c>1$ such that 
for any two $d$-simplexes $\triangle$ and $\triangle'$ in the triangulation $T$, $$\frac1c <\frac{\mathrm{diam}(\triangle)}{\mathrm{diam}(\triangle')}<c\;\;\text{ and }\;\;\frac1c <\frac{\mathrm{diam}(\triangle)}{\mathrm{vol}(\triangle)^{\frac1d}}<c
.$$

%A triangulation $T$  of $M$ is \emph{uniform} if there exist $N>1$ and  $c>1$  such that either the number of vertices of $T$ is smaller than $N$, or $T$  is $c$-uniform. The constants $N$ and $c$ are called the \emph{uniform parameters} of the triangulation.
A set $\mathcal{T}$ of triangulations   of $M$ is \emph{uniform} if there exist $N>1$ and  $c>1$  such that for each triangulation in $\mathcal{T}$, either its  number of vertices  is smaller than $N$, or it  is $c$-uniform. In this case,  these triangulations are simply said to be \emph{uniform}, and the constants $N$ and $c$ are called the \emph{uniform parameters} of these triangulations. 
\end{defn}
\begin{theorem}\label{thm:Cheeger-manifold-complex}
Let $M$ be an  orientable, compact, closed Riemannian manifold of dimension $(d+1)$. %Let  $\Sigma$ be a simplicial complex which is  combinatorially equivalent to a uniform triangulation of $M$. Then, there is a Cheeger inequality
There is a constant $C$ such that for all simplicial complexes $\Sigma$ that are combinatorially equivalent to some prescribed uniform triangulations of $M$, 
\begin{equation}\label{eq:upL-Chee}
\frac{h^2(\Sigma_d)}{C}\le \lambda_{I_{d}}(\Delta_{d}^{up}) \le C\cdot h(\Sigma_d).     
\end{equation} 
%where $C$ is a uniform constant which  is independent of the choice of $\Sigma$. 
In addition, $h(\Sigma_d)>0$ if and only if $H_1(\Sigma)=0$ (or equivalently, $H_1(M)=0$).
\end{theorem}
\begin{proof}
By Proposition \ref{pro:rough-Cheeger}, $\lambda_{I_{d}}(\Delta_{d}^{up})=0$ if and only if  $h(\Sigma_d)=0$. So, it suffices to assume that $h(\Sigma_d)>0$, i.e., $\tilde{H}^d(M)=\tilde{H}^d(\Sigma)=0$. Since $M$ and $\Sigma$ are of  dimension $(d+1)$, Poincar\'e duality implies that  $\tilde{H}_1(M)=\tilde{H}^d(M)=0$.  

Since there are only finitely many   simplicial complexes with less than a given number of vertices, the existence of the constant $C>0$ in \eqref{eq:upL-Chee} for these simplicial complexes  follows immediately from Proposition \ref{pro:rough-Cheeger}. 

%To avoid the above trivial cases, 
In the sequel, we may assume  without loss of generality that $M$ is  simply connected, and the  triangulation is $c$-uniform for some $c>1$, and $\Sigma_d$ has $n$ elements, where $n$ is a sufficiently large integer. 

For any $\epsilon>0$, there exists $N>0$ such that any $c$-uniform triangulation with at least $N$ maximal faces satisfies   $\frac{1}{2c}\epsilon'<\mathrm{diam}(\triangle)<\epsilon'$  for any maximal face $ \triangle$ in the triangulation, where $\epsilon'$ is some small number in $(0,\epsilon)$. 
We shall also regard the uniform triangulation as a uniform  $\epsilon$-net.

With slight abuse of notation, later on, we will use $\epsilon:=\max_{\triangle\in T}\mathrm{diam}(\triangle)$ to indicate the maximum diameter of the maximal faces in the  triangulation. In other words, we fix a  small  $\epsilon>0$ and a sufficiently large $N>0$,  we consider any  simplicial complex $\Sigma$ which is combinatorially equivalent to a $c$-uniform triangulation with at least $N$ maximal faces, and we let  $\epsilon\coloneqq\max_{\sigma\in \Sigma_{d+1}}\mathrm{diam}(|\sigma|).$

\begin{enumerate}
\item[Claim 1] For the down Cheeger constant $h_{down}(\Sigma_{d+1})$, we have
$$ \frac{d+2}{4}h^2_{down}(\Sigma_{d+1})\le \lambda_{I_d}(\Delta_d^{up}) \le (d+2)h_{down}(\Sigma_{d+1}). $$

Proof: This is derived by the Cheeger inequality
$$ \frac{h^2_{down}(\Sigma_{d+1})}{2}\le \lambda_{2}(\Delta_{d+1}^{down}) \le 2h_{down}(\Sigma_{d+1}) $$
proposed in Remark \ref{remark:down-Cheeger}, 
and the equality 
\begin{equation}\label{eq:(un)normal-up/down}\lambda_{I_d}(\Delta_d^{up})=\frac12\lambda_{I_d}(L_d^{up})= \frac12\lambda_{2}(L_{d+1}^{down})=\frac{d+2}{2}\lambda_{2}(\Delta_{d+1}^{down}).\end{equation} 
Here,  $\lambda_{I_d}(L_d^{up})$ and $  \lambda_{2}(L_{d+1}^{down})$  denote the smallest non-zero eigenvalues of the unnormalized up-Laplacian $L_d^{up}$ and the unnormalized down-Laplacian $L_{d+1}^{down}$, respectively.  By \eqref{com2}, we infer the second equality in \eqref{eq:(un)normal-up/down}. Moreover, the first and the last equality in \eqref{eq:(un)normal-up/down} are based on the fact that the degrees appearing in the expression of the $d$-th normalized up-Laplacian $\Delta_d^{up}$ are all equal to 2, while the degrees  in the expression of the $(d+1)$-th normalized down-Laplacian $\Delta_{d+1}^{down}$ are all equal to $d+2$.  In fact, since $M$ is a compact manifold without boundary,  each $d$-face is contained in exactly two  $(d+1)$-faces, and every $(d+1)$-face contains exactly $(d+2)$ different  $d$-faces. 
\item[Claim 2] The Cheeger constant $h(\Sigma_d)$ and the down Cheeger constant  $h_{down}(\Sigma_{d+1})$ satisfy $h(\Sigma_d)\sim  h_{down}(\Sigma_{d+1})$,  
i.e., there exists a uniform constant  $C>1$ such that 
\[\frac1C h_{down}(\Sigma_{d+1})\le h(\Sigma_d)\le C\,  h_{down}(\Sigma_{d+1}).\]

The proof is further divided into the following two claims.
\begin{enumerate}
\item[Claim 2.1] $\frac1\epsilon h_{down}(\Sigma_{d+1})\sim  h(M)$

Proof: Let $G$ be the graph with $n\coloneqq \#\Sigma_{d+1}$ vertices located in the barycenters of all $(d+1)$-simplexes, such that two vertices form an edge in $G$ if and only if these two $d$-simplexes are down adjacent. We may call $G$ the underlying graph of the  triangulation. 

Note that $ h_{down}(\Sigma_{d+1})$ also indicates the Cheeger constant of the  unweighted underlying graph $G$. 
%The graph $G$ generated by $\Sigma_d$ is coarsely equivalent to a   geometric graph with $n$ vertices  investigated  in \cite{TMT20,TrillosSlepcev-16}. 
An approximation  approach developed in \cite{TMT20,TrillosSlepcev-16} implies  that the Cheeger constant of a   uniform triangulation should approximate the
Cheeger constant of the manifold when we equip the edges of the underlying graph of the triangulation with  appropriate  weights (related to $\epsilon$).   In fact, since $G$ is  the underlying graph of the triangulation, we may assume that   $G$ is embedded in the manifold $M$. %and the  distribution of the vertices of $G$ is uniform\footnote{The vertices of $G$ are well-distributed on $M$.}. 
Then, according to the approximation theorems in \cite{TMT20,TrillosSlepcev-16}, by adding  appropriate  weights (related to $\epsilon$)\footnote{The weight of an edge $\{u,v\}$ is determined by the distance of $u$ and $v$ in $M$, which is about $O(\epsilon)$.} on $G$, the Cheeger constant of $G$ (with appropriate  edge weights) would approximate $h(M)$ (i.e., the difference of $h(M)$ and the Cheeger constant of the weighted graph $G$ is bounded by  $h(M)/2$ whenever $\epsilon$ is sufficiently small). 
We can then adopt the same approximation  approach as in
\cite{TMT20,TrillosSlepcev-16} (more precisely, a slight modification of the  approximation theorem 
in \cite{TMT20,TrillosSlepcev-16,TrillosSlepcev-15}) 
 to derive that $ \frac1\epsilon h_{down}(\Sigma_{d+1})\sim  h(M)$.

\item[Claim 2.2] $\frac1\epsilon h(\Sigma_d)\sim  h(M)$ whenever $H_1(M)=0$.

Proof: It is well-known that  $H_1(M)=0$ if and only if $H^d(M)=0$ if and only if $\mathrm{Ker}(\delta_d)=\mathrm{Im}(\delta_{d-1})$,  as $M$ is a compact closed manifold of dimension $(d+1)$. Thus,  %by $\mathrm{Im} (B_d^\top)=\mathrm{Ker} (B_{d+1}^\top)$, $\delta_d=B_{d+1}^\top$, and 
by applying \eqref{eq:B-inf}, we have  $$h(\Sigma_d)= \min\limits_{x\not\in  \mathrm{Ker}(\delta_{d})}\frac{\sum\limits_{\sigma\in \Sigma_{d+1}}\left|\sum\limits_{\tau\in \Sigma_d}\mathrm{sgn}([\tau],\partial[\sigma])x_\tau\right|}{\min\limits_{z\in \mathrm{Ker}(\delta_d)}\sum\limits_{\tau\in \Sigma_d}2|x_\tau+z_\tau|}.$$
By the duality theorems established in \cite{Jost/Zhang21c,TZ22+} (see Section \ref{sub:duality} for an explanation), %(see  Lemma 2.5 and Theorem 2.1  in \cite{Jost/Zhang21c}, or the main theorem in \cite{TZ22+}),   %(see Lemma \ref{lem:dual} and Proposition \ref{pro:p-Lap-dual}), 
we  further obtain
$$h(\Sigma_d)=\min\limits_{y\text{ non-constant}}\frac{\max\limits_{\sigma \mathop{\sim}\limits^{\text{down}} \sigma'}\frac12|y_\sigma-y_{\sigma'}|}{\min\limits_{t\in\R}\max\limits_{\sigma\in \Sigma_{d+1}}|y_\sigma+t|}$$
where $\sigma \mathop{\sim}\limits^{\text{down}} \sigma'$ means $\sigma$ and $\sigma'$ are down adjacent, i.e., they share a common facet. 
And then by elementary techniques, there is no difficulty to check that the  optimization in the right hand side coincides with  
$$\min\limits_{\min\limits_{\sigma} y_\sigma+\max\limits_\sigma y_\sigma=0}\frac{\max\limits_{\sigma \mathop{\sim}\limits^{\text{down}} \sigma'}|y_\sigma-y_{\sigma'}|}{2\max\limits_{\sigma}|y_\sigma|}=\frac{1}{\mathrm{diam}(G)}$$
where $\mathrm{diam}(G)$ indicates the %combinatorial 
diameter of $G$. We remark here that we indeed rewrite   $h(\Sigma_d)$ as the smallest non-trivial eigenvalue of the $\infty$-Laplacian, which agrees with  $1/\mathrm{diam}(G)$.  This argument is similar to a result  in  \cite{Juutinen99}.

Finally, since the  triangulation is $c$-uniform, we obtain 
$$\frac1\epsilon h(\Sigma_d)= \frac{1}{\epsilon\cdot \mathrm{diam}(G)}\sim \frac{1}{ \mathrm{diam}(M)} .$$
Hence,  $\frac1\epsilon h(\Sigma_d)\sim h(M)$.
\end{enumerate}
\end{enumerate}
The proof is then  completed by combining all the statements above.
\end{proof}
\begin{remark}
We conclude this section with some observations regarding the latest theorem. 
\begin{itemize}
\item The constant $C$ in  Theorem \ref{thm:Cheeger-manifold-complex} depends
  on the uniform parameters of the triangulations, and the ambient
  manifold. We hope that it is possible to find a new approach to get a uniform
  constant that only depends on the dimension $d$.
\item Under the same condition as in  Theorem \ref{thm:Cheeger-manifold-complex},
  we further have $$\frac{\lambda_{k_{d}}(\Delta_{d,1}^{up})^2}{C}\le
  \lambda_{k_{d}}(\Delta_{d}^{up}) \le
  C\lambda_{k_{d}}(\Delta_{d,1}^{up}),$$ where $k_d:=\dim \mathrm{Ker}(B_{d+1}^\top)+1$. This inequality  coincides with the
  Cheeger inequality in Theorem \ref{thm:Cheeger-manifold-complex} if and only
  if $H_1(M)=0$.

\item A modification of the proof can deduce that $\frac{1}{\mathrm{diam}(G)}\sim \lambda_2(G)$ whenever $G$ can be uniformly embedded  into such a typical manifold, where $\lambda_2(G)$ is the second smallest  eigenvalue of the normalized  Laplacian on $G$.  
\item Inspired by the approximation theory for Laplacians on triangulations  of manifolds proposed by Dodziuk \cite{Dodziuk76} and  Dodziuk-Patodi \cite{Dodziuk76a}, we hope that it is possible to develop an approximation theory for our Cheeger constants on triangulations of manifolds. 
\end{itemize}
\end{remark}

\subsection{Background
needed for the proof of Theorem \ref{thm:Cheeger-manifold-complex}}
\label{sub:duality}

Since duality is an important ingredient in the proof, we shall give some details on it. 

\begin{lemma}[\cite{Jost/Zhang21c,TZ22+}]\label{Lemma_duality}
Assume we have a linear map $A:\R^n\to\R^m$ and two convex one-homogeneous  functions $\Phi:\R^m\rightarrow [0,+\infty)$ and $\Psi:\R^n\rightarrow [0,+\infty)$ with $\Phi(\vec y)=0$ iff $\vec y=0$ and $\Psi(\vec x)=0$ iff $\vec x=0$.
Then the nonzero  eigenvalues of the following nonlinear eigenproblem \begin{equation}\label{eq:eigen}
 0\in\nabla  \Phi(A\mathbf{x})-\lambda \nabla\Psi( \mathbf{x})  
\end{equation}
and  the nonzero  eigenvalues of the dual eigenproblem
\begin{equation}\label{eq:dual-eigen}
 0\in\nabla  \Psi_*(A^\top\mathbf{y})-\lambda \nabla\Phi_*( \mathbf{y})  
\end{equation}
coincide exactly, 
where $\Phi_*(\mathbf{y})=\sup\limits_{\mathbf{y}'\ne0} \langle \mathbf{y}, \mathbf{y}'\rangle/\Phi(\mathbf{y}')$ and $\Psi_*(\mathbf{x})=\sup\limits_{\mathbf{x}'\ne0} \langle \mathbf{x}, \mathbf{x}'\rangle/\Psi(\mathbf{x}')$. 
\end{lemma}

\begin{lemma}[\cite{Jost/Zhang21c,TZ22+}]\label{Lemma_duality2}
The smallest nonzero eigenvalue of \eqref{eq:eigen} equals
$$
\min_{\mathbf{x}\not\in \mathrm{Ker}(A)}\frac{\Phi(A\mathbf{x})}{\min_{\mathbf{z} \in \mathrm{Ker}(A)}\Psi( \mathbf{x}+\mathbf{z})},
$$
and dually, the smallest nonzero eigenvalue of  \eqref{eq:dual-eigen} is equal to
$$
\min_{\mathbf{y}\not\in \mathrm{Ker}(A^\top)}\frac{\Psi_*(A^\top\mathbf{y})}{\min_{\mathbf{u} \in \mathrm{Ker}(A^\top)}\Phi_*( \mathbf{y}+\mathbf{u})}.
$$
\end{lemma}

Combining Lemmas \ref{Lemma_duality} and \ref{Lemma_duality2} implies 
\begin{equation}\label{eq:important}
\min_{\mathbf{x}\not\in \mathrm{Ker}(A)}\frac{\Phi(A\mathbf{x})}{\min_{\mathbf{z} \in \mathrm{Ker}(A)}\Psi( \mathbf{x}+\mathbf{z})}
=\min_{\mathbf{y}\not\in \mathrm{Ker}(A^\top)}\frac{\Psi_*(A^\top\mathbf{y})}{\min_{\mathbf{u} \in \mathrm{Ker}(A^\top)}\Phi_*( \mathbf{y}+\mathbf{u})}.
\end{equation}

Take $\Phi(\mathbf{y})=\|\mathbf{y}\|_1$ for $\mathbf{y}\in\R^{\Sigma_{d+1}}$,   $\Psi(\mathbf{x})=2\|\mathbf{x}\|_1$ for $\mathbf{x}\in\R^{\Sigma_{d}}$,  and let $A=\delta_d$. Then regarding $A$ as a matrix, the $(\sigma,\tau)$-entry of $A$ is $\mathrm{sgn}([\tau],\partial[\sigma])$, where $\sigma\in \Sigma_{d+1}$ and $\tau\in\Sigma_d$. 
Furthermore,  %the norm duality gives
$\Phi_*(\mathbf{y})=\|\mathbf{y}\|_\infty$, $\Psi_*(\mathbf{x})=\frac12\|\mathbf{x}\|_\infty$, and $|(A^\top \mathbf{y})_\tau|=|y_\sigma-y_{\sigma'}|$ whenever $\tau$ is the common facet of $\sigma$ and $\sigma'$.   The connectivity of the graph on $\Sigma_{d+1}$ equipped with the down-adjacency $\mathop{\sim}\limits^{\text{down}}$  implies $\mathrm{Ker}(A^\top)=\{t\mathbf{1}:t\in\R\}$, where $\mathbf{1}$ is constant 1 on all $\sigma\in \Sigma_{d+1}$. 
Then, \eqref{eq:important}  can be reformulated as:
$$\min\limits_{x\not\in  \mathrm{Ker}(\delta_{d})}\frac{\sum\limits_{\sigma\in \Sigma_{d+1}}\left|\sum\limits_{\tau\in \Sigma_d}\mathrm{sgn}([\tau],\partial[\sigma])x_\tau\right|}{\min\limits_{z\in \mathrm{Ker}(\delta_d)}\sum\limits_{\tau\in \Sigma_d}2|x_\tau+z_\tau|}=
\min\limits_{y\text{ non-constant}}\frac{\max\limits_{\sigma \mathop{\sim}\limits^{\text{down}} \sigma'}\frac12|y_\sigma-y_{\sigma'}|}{\min\limits_{t\in\R}\max\limits_{\sigma\in \Sigma_{d+1}}|y_\sigma+t|}$$
which gives the most important transition in Claim 2.2. 
%\begin{align*}\min\limits_{y\text{ non-constant}}\frac{\max\limits_{\sigma \mathop{\sim}\limits^{\text{down}} \sigma'}\frac12|y_\sigma-y_{\sigma'}|}{\min\limits_{t\in\R}\max\limits_{\sigma\in \Sigma_{d+1}}|y_\sigma+t|}&=\frac12\min\limits_{\substack{\max f=1\\\min f=-1}}\max_{\{v,v'\}\in E}|f(v)-f(v')|\\&=\min\limits_{\min\limits_{\sigma} y_\sigma+\max\limits_\sigma y_\sigma=0}\frac{\max\limits_{\sigma \mathop{\sim}\limits^{\text{down}} \sigma'}|y_\sigma-y_{\sigma'}|}{2\max\limits_{\sigma}|y_\sigma|}=\frac{1}{\mathrm{diam}(G)}   \end{align*}

\section{Cheeger-type  inequalities for $p$-Laplacians on simplicial complexes} 
\label{sec:p-lap}

In this section, we shall study the nonlinear eigenvalue problems for the $p$-Laplacians on simplicial complexes introduced in Section \ref{plap}. Importantly, this will provide a perspective to  unify some  Cheeger-type inequalities.

According to the main theorem in \cite{TZ22+}, the  spectral duality that we
had used for the $2$-Laplacian now becomes
\begin{pro}\label{pro:nonzero-p-Lap}
The nonzero
eigenvalues of the up $p$-Laplacians are in  one-to-one correspondence  with those of the 
down  $p^*$-Laplacians:
$$\{\lambda^{\frac 1p}:\lambda\text{ is a nonzero eigenvalue of }L_{d,p}^{up}\}=\{\lambda^{\frac {1}{p^*}}:\lambda\text{ is a nonzero eigenvalue of }L_{d+1,p^*}^{down}\},$$
where $p^*$ is the H\"older conjugate of $p$, i.e., $\frac1p+\frac{1}{p^*}=1$. 
Moreover, $\lambda^{\frac 1p}_{n-i}(L_{d,p}^{up})=\lambda^{\frac {1}{p^*}}_{m-i}(L_{d+1,p^*}^{down})$ for any $i=0,1,\cdots,\min\{n,m\}-1$, where $n=|\Sigma_d|$ and $m=|\Sigma_{d+1}|$. 
\end{pro}

The case  $p=2$ of   Proposition \ref{pro:nonzero-p-Lap}
is of course the well-known relation between up and down  Laplacians that we
had already noted in Section \ref{sec:Sim-complex}, that is, the nonzero
eigenvalues of $L^{up}_d$ and  $L^{down}_{d+1}$ coincide. 

So, we can concentrate on the up $p$-Laplacian  for investigating the
spectra of simplicial complexes. To get more concise %and more useful 
results,  we will work with the {\sl normalized up $p$-Laplace operator}
$\Delta^{up}_{d,p}$, whose eigenvalues are determined by the critical values 
of the $p$-Rayleigh quotient
\[f\mapsto\frac{\|B_{d+1}^\top f\|_p^p}{ \|f\|_{p,\deg}^p}\]
where $\|f\|_{p,\deg}^p=\sum_{\tau\in \Sigma_d}\deg \tau\cdot |f(\tau)|^p$. {Similar to \eqref{plap1}, we shall  focus on the min-max eigenvalues
\begin{equation}
\label{plap1+}\lambda_i(\Delta^{up}_{d,p}):=\inf_{\gamma(S)\ge i}\sup_{f\in
  S}\frac{\|B_{d+1}^\top f\|_p^p}{ \|f\|_{p,\deg}^p},\;\;i=1,2,\cdots,n.
\end{equation}
}

\begin{theorem}\label{thm:anti-signed-Cheeger-p}
For any simplicial complex and every $d\ge 0$, for any $p\in (1,+\infty)$, there exist uniform positive constants $C_{p,d}$ and $c_{p,d}$  such that
\begin{align*}\label{eq:Cheeger-1-complex-p}
c_{p,d}h_1(\Sigma_d)^p&\le (d+2)^{p-1}-\lambda_n(\Delta^{up}_{d,p})\;\text{ when }p\ge2
\\
C_{p,d}h_1(\Sigma_d)&\ge (d+2)^{p-1}-\lambda_n(\Delta^{up}_{d,p})\;\text{ when }1<p\le2
\end{align*}
where $n=|\Sigma_d|$. 

\end{theorem}

\begin{proof}
We need the following key claim.

\textbf{Claim}.  For any $p>1$, and for any integer $k\ge 2$, there exist $m_{p,k}>0$ (if $ p\ge 2$) and $M_{p,k} >0$ (if $1<p\le  2$) such that, for any $\mathbf{x}\in\R^k$,
\begin{align*}
m_{p,k}{\sum_{1\le i<j\le k}|x_i-x_j|^p}&\le {k^{p-1}\sum_{i=1}^k|x_i|^p-\left|\sum_{i=1}^kx_i\right|^p} \; \text{  if }p\ge 2  ,\\
M_{p,k}{\sum_{1\le i<j\le k}|x_i-x_j|^p}&\ge {k^{p-1}\sum_{i=1}^k|x_i|^p-\left|\sum_{i=1}^kx_i\right|^p}\; \text{  if }1<p\le 2 .
\end{align*}

\textbf{Proof}. It suffices to prove that 
$$m_{p,k}'\coloneqq \inf\limits_{x\text{ non-constant}}\frac{k^{p-1}\sum_{i=1}^k|x_i|^p-|\sum_{i=1}^kx_i|^p}{\sum_{i,j=1}^k|x_i-x_j|^p}>0$$
for $p\ge2$, and
$$M_{p,k}'\coloneqq \sup\limits_{x\text{ non-constant}}\frac{k^{p-1}\sum_{i=1}^k|x_i|^p-|\sum_{i=1}^kx_i|^p}{\sum_{i,j=1}^k|x_i-x_j|^p}<+\infty$$
for $1<p\le 2$. 
Clearly, $\sum_{i,j=1}^k|x_i-x_j|^p>0$ if and only if  $\mathbf{x}$ is non-constant. By H\"older's inequality, $k^{p-1}\sum_{i=1}^k|x_i|^p\ge|\sum_{i=1}^kx_i|^p$ and equality holds if and only if $\mathbf{x}$ is constant. Therefore,  $\sum_{i,j=1}^k|x_i-x_j|^p>0$ if and only if $k^{p-1}\sum_{i=1}^k|x_i|^p-|\sum_{i=1}^kx_i|^p>0$. Now, given $\delta>0$, for any vector $\mathbf{x}$ satisfying $\max\limits_i x_i-\min\limits_i x_i=\delta$, we have $\delta^p\le \sum_{i,j=1}^k|x_i-x_j|^p\le k^2 \delta^p$. 

Let $g(\mathbf{x},t,p)=k^{p-1}\sum_{i=1}^k|x_i+t|^p-|\sum_{i=1}^kx_i+kt|^p$, for $\mathbf{x}\bot\mathbf{1}$ with $\mathbf{x}\ne \mathbf{0}$, $t\in\R$ and $p\ge1$. 

Since $\mathbf{x}$ is non-constant and $p>1$, by H\"older's inequality, we have $g(\mathbf{x},t,p)>0$. Note that $\partial_t g(\mathbf{x},t,p)=pk^{p-1}\sum_{i=1}^k|x_i+t|^{p-1}\mathrm{sign}(x_i+t)-pk|kt|^{p-1}\mathrm{sign}(kt)$, where we used the assumption $\mathbf{x}\bot\mathbf{1}$, that is, $\sum_{i=1}^kx_i=0$. Fix $p>2$. If $t>k\delta$, by H\"older's inequality, $\partial_t g(\mathbf{x},t,p)>0$. Similarly, if $t<-k\delta$, %by H\"older's inequality,
then $\partial_t g(\mathbf{x},t,p)<0$. Therefore, $t\mapsto g(\mathbf{x},t,p)$ reaches its minimum at some point in $[-k\delta,k\delta]$, implying that 
$$\min\limits_{t\in\R}g(\mathbf{x},t,p)=\min\limits_{-k\delta\le t\le k\delta}g(\mathbf{x},t,p)$$ is a continuous function on the compact set $$\Big\{\mathbf{x}\in\R^k:\sum_{i=1}^k x_i=0,\max\limits_i x_i-\min\limits_i x_i=\delta\Big\}.$$ 
Hence, $\min\limits_{\mathbf{x}\bot \mathbf{1},\max\limits_i x_i-\min\limits_i x_i=\delta}\min\limits_{t\in\R}g(\mathbf{x},t,p)>0$. Thus, 
\begin{align*}
 \inf\limits_{x\text{ non-constant}}\frac{k^{p-1}\sum_{i=1}^k|x_i|^p-|\sum_{i=1}^kx_i|^p}{\sum_{i,j=1}^k|x_i-x_j|^p}&=\inf\limits_{\mathbf{x}\bot\mathbf{1},\max\limits_i x_i-\min\limits_i x_i=\delta}\min\limits_{t\in\R}\frac{g(\mathbf{x},t,p)}{\sum_{i,j=1}^k|x_i-x_j|^p}  
 \\&\ge \frac{1}{k^2\delta^p}\min\limits_{\mathbf{x}\bot\mathbf{1},\max\limits_i x_i-\min\limits_i x_i=\delta}\min\limits_{t\in\R}g(\mathbf{x},t,p)>0.
\end{align*}

Now, $g(\mathbf{x},t,2)=\sum_{\{i,j\}\subset\{1,\ldots,k\}}(x_i-x_j)^2$ and $g(\mathbf{x},t,1)\ge0$.%$g(\mathbf{x},t,1)=0$ for $t\in\R\setminus(-k\delta,k\delta)$.

Furthermore, $\partial_p g(\mathbf{x},t,p)=\frac1k \sum_{i=1}^k|kx_i+kt|^p\ln|kx_i+kt|-|\sum_{i=1}^kx_i+kt|^p\ln|\sum_{i=1}^kx_i+kt|$. Since $s\mapsto s^p\ln s$ is a convex and increasing function on $s\in(1,+\infty)$, by Jensen's inequality, we have $\partial_p g(\mathbf{x},t,p)>0$ whenever $|t|>\delta+1/k$ and $p>1$. Therefore, 
$$g(\mathbf{x},t,1)< g(\mathbf{x},t,p)\le \sum_{\{i,j\}\subset\{1,\ldots,k\}}(x_i-x_j)^2$$
whenever $|t|>\delta+1/k$ and $1<p\le 2$. Consequently, \begin{align*}
 \sup\limits_{x\text{ non-constant}}\frac{k^{p-1}\sum_{i=1}^k|x_i|^p-|\sum_{i=1}^kx_i|^p}{\sum_{i,j=1}^k|x_i-x_j|^p}&=\sup\limits_{\mathbf{x}\bot\mathbf{1},\max\limits_i x_i-\min\limits_i x_i=\delta}\max\limits_{t\in\R}\frac{g(\mathbf{x},t,p)}{\sum_{i,j=1}^k|x_i-x_j|^p}  
 \\&\le \frac{1}{\delta^p}\max\limits_{\mathbf{x}\bot\mathbf{1},\max\limits_i x_i-\min\limits_i x_i=\delta}\max\limits_{t\in\R}g(\mathbf{x},t,p)<+\infty.
\end{align*}
The proof of the claim is completed. 

\vspace{0.3cm}
Now we apply the above claim to estimate the spectral gap of $\lambda_n(\Delta^{up}_{d,p})$ from $(d+2)^{p-1}$. Note that if $1<p\le 2$,
\begin{align*}
&(d+2)^{p-1}-\lambda_n(\Delta^{up}_{d,p})\\=~&(d+2)^{p-1}-\sup\limits_{f\ne0} \frac{\sum\limits_{\sigma\in \Sigma_{d+1}}\left|\sum_{\tau\in \Sigma_d,\tau\subset\sigma}\mathrm{sgn}([\tau],\partial[\sigma])f(\tau)\right|^p}{\sum_{\tau\in \Sigma_d} \deg\tau \cdot |f(\tau)|^p}
\\=~&\inf\limits_{f\ne0} \frac{(d+2)^{p-1}\sum_{\tau\in \Sigma_d} \deg\tau \cdot |f(\tau)|^p-\sum\limits_{\sigma\in \Sigma_{d+1}}\left|\sum_{\tau\in \Sigma_d,\tau\subset\sigma}\mathrm{sgn}([\tau],\partial[\sigma])f(\tau)\right|^p}{\sum_{\tau\in \Sigma_d} \deg\tau \cdot |f(\tau)|^p}
\\=~&\inf\limits_{f\ne0}\frac{(d+2)^{p-1}\sum\limits_{\sigma\in \Sigma_{d+1}}\sum\limits_{\tau\in \Sigma_d,\tau\subset \sigma}|f(\tau)|^p-\sum\limits_{\sigma\in \Sigma_{d+1}}\left|\sum_{\tau\in \Sigma_d,\tau\subset\sigma}\mathrm{sgn}([\tau],\partial[\sigma])f(\tau)\right|^p}{\sum_{\tau\in \Sigma_d} \deg\tau \cdot |f(\tau)|^p}
\\=~&\inf\limits_{f\ne0}\frac{\sum\limits_{\sigma\in \Sigma_{d+1}}\left((d+2)^{p-1}\sum\limits_{\tau\in \Sigma_d,\tau\subset \sigma}|f(\tau)|^p- \left|\sum_{\tau\in \Sigma_d,\tau\subset\sigma}\mathrm{sgn}([\tau],\partial[\sigma])f(\tau)\right|^p\right)}{\sum_{\tau\in \Sigma_d} \deg\tau \cdot |f(\tau)|^p}
\\\le~&\inf\limits_{f\ne0}\frac{\sum\limits_{\sigma\in \Sigma_{d+1}}M_{p,d+2}\sum\limits_{{\tau,\tau'\in \Sigma_d,\tau,\tau'\subset\sigma}}  |\mathrm{sgn}([\tau],\partial[\sigma])f(\tau)-\mathrm{sgn}([\tau'],\partial[\sigma])f(\tau') |^p}{\sum_{\tau\in \Sigma_d} \deg\tau \cdot |f(\tau)|^p}
\\=~&M_{p,d+2}(d+1)\inf\limits_{f\ne0}\frac{\sum\limits_{{\tau,\tau'\in \Sigma_d,\tau\sim^d\tau'}}  |f(\tau)-s([\tau'],[\tau])f(\tau') |^p}{\sum_{\tau\in \Sigma_d} \widetilde{\deg}\tau \cdot |f(\tau)|^p}
\\=~&M_{p,d+2}(d+1) \lambda_1
(\Delta_p(\Gamma_d,s))\le M_{p,d+2}(d+1)2^{p-1} h(\Gamma_d,s)=M_{p,d+2}2^{p-1} h_1(\Sigma_d)
\end{align*}
where  $\widetilde{\deg}\,\tau=(d+1)\deg\tau$ is the degree of $\tau$ in  $(\Gamma_d,s)$. Similarly, if $p\ge 2$,
\begin{align*}
(d+2)^{p-1}-\lambda_n(\Delta^{up}_{d,p})&\ge m_{p,d+2}(d+1)\lambda_1(\Delta_p(\Gamma_d,s))
\\&\ge m_{p,d+2}(d+1)2^{p-1}\frac{h^p(\Gamma_d,s)}{p^p}=m_{p,d+2} \frac{h^p}{p^p}\left(\frac{2}{d+1}\right)^{p-1}    
\end{align*}
where we used a Cheeger inequality for the  $p$-Laplacian on signed graphs from \cite{Amghibech,GLZ22}. 

Therefore,  we can always take $$c_{p,d}=
\frac{m_{p,d+2}2^{p-1}}{p^p(d+1)^{p-1}}\text{ for }p\ge 2\;\;\text{ and }\;
C_{p,d}=2^{p-1}M_{p,d+2}\text{ for }1<p\le 2.$$ While, for the case of $p=2$ (already treated in Theorem \ref{thm:anti-signed-Cheeger}), it follows from 
$$ k\sum_{i=1}^kx_i^2-\left(\sum_{i=1}^kx_i\right)^2=\sum_{1\le i<j\le k}(x_i-x_j)^2$$
that we can choose  $m_{2,k}=M_{2,k}=1$, for any $k\ge2$, and $c_{2,d}=\frac{1}{2(d+1)}$ and $C_{2,d}=2$ for any $d\ge0$. 
\end{proof}

\begin{remark}\label{rem:k-way-p}
In fact, we can further prove that  there exist   absolute constants $C_{p,d}'>0$ and $ c_{p,d}'>0$  such that for any simplicial complex, and for any $k\ge 1$, 
\begin{align*}%\label{eq:Cheeger-k-complex-p}
 (d+2)^{p-1}-\lambda_{n+1-k}'(\Delta^{up}_{d,p})&\le C_{p,d}'h_k(\Sigma_{d}),\;\text{ if }p\in (1,2],\\
(d+2)^{p-1}-\lambda_{n+1-k}'(\Delta^{up}_{d,p})&\ge c_{p,d}'\frac{ h_k(\Sigma_d)^p}{k^{3p}},\;\text{ if }p\in [2,+\infty),
\end{align*}
where \[\lambda_{n+1-k}'(\Delta^{up}_{d,p}):=\sup_{\gamma(S)\ge k}\inf_{f\in S} \frac{\sum\limits_{\sigma\in \Sigma_{d+1}}\left|\sum_{\tau\in \Sigma_d,\tau\subset\sigma}\mathrm{sgn}([\tau],\partial[\sigma])f(\tau)\right|^p}{\sum_{\tau\in \Sigma_d} \deg\tau \cdot |f(\tau)|^p}\]
indicates the $(n+1-k)$-th max-min eigenvalue. %$k$-th largest max-min eigenvalue. 
Clearly, $\lambda_n'(\Delta^{up}_{d,p})=\lambda_n(\Delta^{up}_{d,p})$ for any $p$.

For simplicity, and to avoid tedious processes, we just sketch the proof below. 
First, using the claim in the proof of Theorem \ref{thm:anti-signed-Cheeger-p}, we have
\begin{align*}
(d+2)^{p-1}-\lambda_{n+1-k}'(\Delta^{up}_{d,p})&\le (d+1)M_{p,d}\lambda_{k}(\Delta_p(\Gamma_d,s))    ,\;\text{ if }p\in (1,2],\\
(d+2)^{p-1}-\lambda_{n+1-k}'(\Delta^{up}_{d,p})&
\ge (d+1)m_{p,d}\lambda_{k}(\Delta_p(\Gamma_d,s)),\;\text{ if }p\in [2,+\infty),
\end{align*}
where $\Delta_p(\Gamma_d,s)$ represents the $p$-Laplacian on the signed graph $(\Gamma_d,s)$. 
By a slightly modified variant 
of Theorem 1.4 in \cite{Zhang21}, and by Theorem \ref{thm:signed-Cheeger}, we derive that for any $p\in[1,+\infty)$, \[\lambda_{k}(\Delta_p(\Gamma_d,s))\le 2^{p-1} \frac{h_k(\Sigma_d)}{d+1}
\]
 and, for any $p\ge 2$, 
\[\lambda_{k}(\Delta_p(\Gamma_d,s))\ge 2^{\frac p2-1}\left(\frac 2p\right)^p \lambda_{k}(\Delta_p(\Gamma_d,s))^{\frac p2}\ge  2^{\frac p2-1}\left(\frac 2p\right)^p \left(\frac{1}{Ck^6}\Big(\frac{h_k(\Sigma_d)}{d+1}\Big)^2\right)^{\frac p2}.\] 
%in a similar manner. 
The proof is then completed by combining the above  inequalities.
\end{remark}

\begin{remark}
Theorem \ref{thm:anti-signed-Cheeger} can be recovered by taking $p=2$ in Theorem \ref{thm:anti-signed-Cheeger-p} and Remark \ref{rem:k-way-p}. 
%The case of $p=2$ reduces to . 
\end{remark}

The last result gives a nonlinear version of the main theorem in Section \ref{sec:gap-0}. 
We put
$$\lambda_{I_d}(\Delta^{up}_{d,p})=\min\limits_{x\bot \mathrm{Image}(B_d^\top)}\frac{\|B_{d+1}^\top\vec x\|_p^p}{\min\limits_{y\in \mathrm{Image}(B_d^\top)}\|\vec x+\vec y\|_{p,\deg}^p}$$ which indicates the first nontrivial eigenvalue of $\Delta^{up}_{d,p}$.

\begin{pro}\label{pro:rough-Cheeger-p}
Suppose that $\deg \tau>0$, $\forall \tau\in \Sigma_d$. Then,  for any $p\ge 1$,
$$\frac{h^p(\Sigma_d)}{|\Sigma_{d+1}|^{p-1}}\le \lambda_{I_d}(\Delta_{d,p}^{up})\le \vol(\Sigma_d)^{p-1}h(\Sigma_d).$$
\end{pro}
\begin{proof}
     %The case of $\Delta_{d,p}^{up}$ is similar. 
     The proof is easy and  very similar to Proposition \ref{pro:rough-Cheeger}.
\end{proof}

\begin{theorem}\label{thm:p-lap-Cheeger}
Let $M$ be an  orientable, compact, closed Riemannian manifold of dimension $(d+1)$. %Let  $\Sigma$ be a simplicial complex which is  combinatorially equivalent to a uniform triangulation of $M$. 
There exists a constant $C$ such that for all simplicial complexes $\Sigma$ that are combinatorially equivalent to some given uniform triangulations of $M$,
%Then, there is a Cheeger inequality
$$ \frac{h^p(\Sigma_d)}{C}\le \lambda_{I_{d}}(\Delta_{d,p}^{up}) \le C\cdot h^{p-1}(\Sigma_d). $$
%where $C$ is a uniform constant which  is independent of the choice of $\Sigma$. 
\end{theorem}

\begin{proof}
The proof is essentially the same to that of Theorem \ref{thm:Cheeger-manifold-complex}, with only a small difference at Claim 1 in the proof of Theorem \ref{thm:Cheeger-manifold-complex}. 
In fact, we only need to use the following claim instead of Claim 1. 

Claim: For the down Cheeger constant, for any $p>1$, we have
\[(d+2)^{p-1}\left(\frac{h_{down}}{p^*}\right)^{p}\le \lambda_{I_d}(\Delta_{d,p}^{up})\le (d+2)^{p-1}h_{down}^{p-1}\]
and in particular, when $p$ tends to $+\infty$, we have 
$\lim\limits_{p\to+\infty}\lambda_{I_d}(\Delta_{d,p}^{up})^{\frac1p}=(d+2)h_{down}$.

Proof: Since the down adjacency relation  induces a graph on $\Sigma_d$, we can directly use the Cheeger inequality for $p$-Laplacian on graphs to derive
\begin{equation}\label{eq:down-Cheeger}
 \frac{2^{p-1}h^p_{down}(\Sigma_{d+1})}{p^p}\le \lambda_{2}(\Delta_{d+1,p}^{down}) \le 2^{p-1}h_{down}(\Sigma_{d+1}) .   
\end{equation}

Since $|\Sigma|$ is a compact piecewise flat  manifold without boundary, and since the dimension of  $|\Sigma|$ is $d+1$, the normalized and unnormalized versions of $p$-Laplacian on $\Sigma$ satisfy 
$\lambda_{i}(L_{d+1,p}^{down})=(d+2)\lambda_{i}(\Delta_{d+1,p}^{down})$ and $\lambda_{i}(L_{d,p}^{up})=2\lambda_{i}(\Delta_{d,p}^{up})$ for any $i$.

Together with the spectral duality 
 $\big(\lambda_{I_d}(L_{d,p^*}^{up})\big)^{\frac{1}{p^*}}=\big(\lambda_{2}(L_{d+1,p}^{down})\big)^{\frac1p}$ derived by Proposition \ref{pro:nonzero-p-Lap}, 
we immediately obtain the duality equality
\begin{equation}\label{eq:dual-equality}\left(2\lambda_{I_d}(\Delta_{d,p^*}^{up})\right)^{\frac{1}{p^*}}=\left((d+2)\lambda_{2}(\Delta_{d+1,p}^{down})\right)^{\frac1p}.\end{equation} Then substituting the duality equality \eqref{eq:dual-equality}  into the above down  Cheeger inequality \eqref{eq:down-Cheeger}, we finally deduce that 
\[(d+2)^{p^*-1}\left(\frac{h_{down}}{p}\right)^{p^*}\le \lambda_{I_d}(\Delta_{d,p^*}^{up})\le (d+2)^{p^*-1}h_{down}^{p^*-1}\]
The proof of the claim is completed by exchanging the positions of $p$ and $p^*$.

Finally, combining the above claim with Claim 2 in the proof of Theorem \ref{thm:Cheeger-manifold-complex}, we derive the desired Cheeger-type inequality stated in Theorem \ref{thm:p-lap-Cheeger}. 
\end{proof}

 \section{Discussions and Remarks}

In the present paper, we have  established several higher-dimensional analogues of the Cheeger inequality from different perspectives.

Although these results appear to be formally similar to the classical one, there are several key differences in both the results and the approaches, as we shall now discuss. 

In Proposition 2.1, there are some factors like $|\Sigma_d|$ and $\vol(\Sigma_d)$, which don't appear in the usual Cheeger inequality. 
These factors are not sharp, but they indeed show that the first nontrivial eigenvalue of the up-Laplacian has the same vanishing condition as the Cheeger constant $h(\Sigma_d)$ we introduced. 
To the best of our knowledge,  $h=h(\Sigma_d)$ is the only geometric quantity that satisfies a Cheeger-type  inequality  of the form $ch^a\le \lambda_{I_d}(L_d^{up})\le Ch^b$. 
%From this 
And more importantly, when we restrict to the $d$-th up-Laplacian on triangulations of  $(d+1)$-manifolds, the factors will be universal. We present the result in Theorem \ref{thm:Cheeger-manifold-complex}, which is  the main result of this paper. 

We  highlight  that the proof of Theorem \ref{thm:Cheeger-manifold-complex} relies heavily on the  approximation theory of Cheeger cuts from refined  triangulations to manifolds. Such an  approximation theory  was  recently  developed by Trillos, Slepcev, et al \cite{{TSBLB-16},TrillosSlepcev-16,TrillosSlepcev-15,TMT20}. 
However,  having this ingredient is still not enough; in fact, our proof further  relies on a  duality argument by considering the smallest  nonzero eigenvalues of both the 1-Laplacian on a triangulation and the $\infty$-Laplacian on its dual graph, respectively. 
 Such a nonlinear spectral duality is similar to the linear case on the relation between up- and down- Laplacians on simplicial complexes, but the nonlinear case requires more subtle techniques, which have  recently been discovered and established in the authors' recent work \cite{Jost/Zhang21c,TZ22+}. 

Based on the original idea of combining the modern approximation theory and our nonlinear spectral duality principle, we further obtain a $p$-Laplacian version of the Cheeger inequality on typical  simplicial complexes. 
%Theorem 2.3 applies to the $d$-Laplacian on a $(d+1)$-manifold. 

%{\color{red}We shall point out more novel fundamental ideas (e.g.,  a conjectural extension to the Riemmanian settings).}

~\\
J\"urgen Jost\\
Max Planck Institute for Mathematics in the Sciences, %Inselstrasse 22, 
04103 Leipzig,
Germany~\\ Email address:
{\tt  jost@mis.mpg.de}
~\\

~\\Dong Zhang\\
 LMAM and School of Mathematical Sciences, 
        Peking University,  
      100871 Beijing, China
\\
Email address: %es:
{\tt dongzhang@math.pku.edu.cn}

\end{document}